\documentclass[11pt]{article}

\usepackage{amsmath,amssymb,amsthm,color}
\usepackage{algorithm,algorithmicx,algpseudocode}

\usepackage[hidelinks,colorlinks=true,linkcolor=blue,citecolor=black]{hyperref}
	
\usepackage{fancybox,graphicx,bm,float}               

\newtheorem{definition}{Definition}

\newtheorem{theorem}{Theorem}
\newtheorem{lemma}[theorem]{Lemma}
\newtheorem{corollary}[theorem]{Corollary}

\newtheorem{example}{Example}

\newcommand{\R}{\mathbb{R}}

\newcommand{\cE}{\mathcal{E}}

\newcommand{\cV}{\mathcal{V}}

\newcommand{\vect}{\text{vec}}
\newcommand{\bpmat}{\begin{pmatrix}}
\newcommand{\epmat}{\end{pmatrix}}
\newcommand{\Symn}{\mathcal{S}^{n \times n}}

\newcommand{\bx}{\bm{x}}

\newcommand{\bT}{\bm{T}}

\newcommand{\bZ}{\bm{Z}}

\newcommand{\bb}{\bm{b}}

\newcommand{\bz}{\bm{z}}
\newcommand{\bX}{\bm{X}}
\newcommand{\bY}{\bm{Y}}
\newcommand{\bE}{\bm{E}}
\newcommand{\bF}{\bm{F}}
\newcommand{\bB}{\bm{B}}

\newcommand{\bR}{\bm{R}}
\newcommand{\bQ}{\bm{Q}}
\newcommand{\bG}{\bm{G}}
\newcommand{\bK}{\bm{K}}
\newcommand{\bH}{\bm{H}}
\newcommand{\bD}{\bm{D}}
\newcommand{\bA}{\bm{A}}
\newcommand{\bU}{\bm{U}}

\newcommand{\bLambda}{\bm{\Lambda}}
\newcommand{\bnabla}{\bm{\nabla}}

\newcommand{\Diag}{\operatorname{Diag}}
\newcommand{\trace}{\operatorname{tr}}
\DeclareMathOperator*{\argmin}{argmin}

\newcommand{\Rmnum}[1]{\uppercase\expandafter{\romannumeral #1}} 
\usepackage[raggedright]{titlesec}
\titleformat{\chapter}{\centering\Huge\bfseries}{Chapter \Rmnum{\thechapter} }{1em}{} 

\usepackage{mathrsfs} 
\usepackage{enumerate} 

\usepackage{multirow}
\usepackage{stmaryrd}

\title{Cubic-Regularized Newton for Spectral Constrained Matrix Optimization and its Application to Fairness }

\author{ Casey Garner\thanks{School of Mathematics, University of Minnesota (\href{mailto:garne214@umn.edu}{garne214@umn.edu}, \href{mailto:lerman@umn.edu}{lerman@umn.edu}) }
\hspace{1cm}
Gilad Lerman\footnotemark[1]
\hspace{1cm}
Shuzhong Zhang\thanks{Department of Industrial and Systems Engineering (\href{mailto:zhangs@umn.edu}{zhangs@umn.edu})}}

\date{\today}

\begin{document}

\maketitle

\vspace{-0.2in}
\begin{abstract}
{Matrix functions are utilized to rewrite smooth spectral constrained matrix optimization problems as smooth unconstrained problems over the set of symmetric matrices which are then solved via the cubic-regularized Newton method. A second-order chain rule identity for matrix functions is proven to compute the higher-order derivatives to implement cubic-regularized Newton, and a new convergence analysis is provided for cubic-regularized Newton for matrix vector spaces. We demonstrate the applicability of our approach by conducting numerical experiments on both synthetic and real datasets. In our experiments, we formulate a new model for estimating fair and robust covariance matrices in the spirit of the Tyler's M-estimator (TME) model and demonstrate its advantage.

\vspace{3mm}

\vspace{3mm}
}
\end{abstract}

\maketitle

\section{Introduction}
The key focus in this paper is the utilization of non-linear matrix transformations to remodel spectrally constrained matrix optimization problems as unconstrained optimization problems over matrix vector spaces. The primary model we consider is,
\begin{align}\label{eqn:main_model}
\min&\;\; F(\bm{X}) \\
\text{s.t.}&\;\; \lambda_i(\bm{X}) \in \mathcal{I}, \;\; i=1,\hdots, n \nonumber \\
			&\;\; \bm{X} \in \Symn, \nonumber 
\end{align}
where $\Symn$ is the set of $n\times n$ real symmetric matrices, $F: \Symn \rightarrow \R$ is two-times continuously differentiable, $\mathcal{I} \subseteq \R$ is an interval, and $\lambda_i(\bX)$ is the $i^\text{th}$ largest eigenvalue of $\bX$. Many important models are captured by \eqref{eqn:main_model} including smooth optimization over symmetric positive definite and positive semidefinite matrices. In this paper we demonstrate how to rewrite \eqref{eqn:main_model} using matrix functions \cite{H08,HJ} as the unconstrained model, 
\begin{align}\label{eqn:unc_main_model}
\min&\;\; F(\bG(\bX)) \\
\text{s.t.}&\;\; \bX \in \Symn, \nonumber 
\end{align}
where $\bG: \Symn \rightarrow \Symn$ and $F\circ \bG : \Symn \rightarrow \R$ are both twice continuously differentiable. We then prove the necessary results for one to apply cubic-regularized Newton to the smooth non-convex model \eqref{eqn:unc_main_model}, and we apply the proposed solution procedure to solve a fair data science model for robust covariance matrix estimation. 

Nesterov and Polyak in \cite{Nes06} proved a global convergence result for a regularized Newton algorithm for an unconstrained, non-convex optimization model. The main step in their approach required solving the subproblem, 
$$
\min_{\bx \in \R^n} \; \bb^\top \bx + \frac{1}{2} \bx^\top \bm{A} \bx + \frac{M}{6}||\bx||^3,
$$
where $\bm{A} \in \mathcal{S}^{n \times n}$ and $M>0$. This cubic-regularized subproblem has proven extremely effective in numerous algorithms and has well-defined conditions ensuring global optimality \cite{Ct11a,cristofari2019global,griewank1981modification,Nes06}. From their work much research has spawned in the past decade and a half. Accelerated cubic-regularized Newton variants \cite{Nes08} and alterations handling constrained problems \cite{benson2014interior,cartis2012adaptive,nesterov2006cubic} have been developed. In recent years cubic-regularized Newton has been adapted to the matrix manifold setting \cite{Ag21,Jun18}.

The main contributions of this work add to the discussions around cubic-regularized Newton, matrix functions and fair data science models. First, we provide a detailed and systematic manner for rewriting \eqref{eqn:main_model} as an unconstrained matrix optimization model through the use of matrix functions. This enables higher-ordered methods such as cubic-regularized Newton to be applied to an equivalent unconstrained optimization model without the use of barrier functions or penalty approaches. Matrix functions have many uses in various areas of mathematics including: differential equations, Markov models, control theory and nonlinear matrix equations \cite{H08}; however, to the authors' knowledge, no discussion of utilizing matrix functions to handle spectral constraints for matrix optimization is present in the literature.

Second, we present a chain rule identity for functions of the form $F \circ \bG$ where $\bG: \cV \rightarrow \cV$ is a matrix function formed from an analytic scalar function (see Theorem \ref{thm:CR}). This identity enables us to compute the second-order directional derivative of $F\circ \bG$ in terms of the first and second-order directional derivatives of $F$ and $\bG$ which are more readily available. We are unaware of this identity in the literature. Derivative results exist for matrix functions \cite{H08,HJ,mathias1996chain}; however, they either do not discuss higher-order chain rule identities or they only present higher-order derivatives for maps taking scalars to matrices. In \cite{sendov2007higher} the authors prove a somewhat similar identity to ours, but for a very restricted class of  functions. 

Third, we provide a new analysis of the convergence of cubic-regularized Newton which directly extends Nesterov and Polyak's original analysis in \cite{Nes06} into the matrix vector space setting. Our analysis relies on vectorizing the matrix subproblem and studying its relevant properties. Thus, we avoid introducing manifold geometry as in \cite{Jun18}.

At last, another contribution of this paper is the Fair Tyler's M-estimator (Fair TME) model \eqref{eqn:fair_TME} we propose and solve in Section \ref{sec:experiments}. The Tyler's M-estimator model \cite{T87} seeks to estimate a robust covariance matrix for a dataset which might contain outliers. This model and solution approaches for it have been studied recently and extensively \cite{danon2022frank,D16,F20,G20,wiesel2015structured}. Additionally, in recent years the discussion of fair models in data science and machine learning has increased \cite{mehrabi2021survey} leading to the development of models like fair PCA \cite{samadi2018price}. The Fair TME model we develop seeks to find a robust covariance matrix that fairly represents different subgroups of a single dataset. We test our fair model on both synthetic and real datasets and demonstrate its superiority for fairness when compared to the standard TME model applied to an aggregation of the datasets.

The paper is organized in the following manner: Section \ref{sec:matrix_functions} defines matrix functions and demonstrates how they are utilized to obtain unconstrained optimization models; Section \ref{sec:matrix_calculus} provides an overview of the fundamentals of matrix calculus and presents a new second-order chain rule identity for matrix functions; Section \ref{sec:CRN_for_matrix_opt} presents the cubic-regularized Newton algorithm and states the global $\mathcal{O}(\epsilon^{-3/2})$ convergence and local superlinear convergence results in the matrix optimization setting; Section \ref{sec:implement} discusses the implementation of cubic-regularized Newton; Section \ref{sec:experiments} showcases the numerical experiments for the Fair TME model on synthetic and real data; Section \ref{sec:conclusions} concludes the main body of the paper and outlines future avenues of research, and Sections \ref{sec:Appendix_B}, \ref{sec:Appendix_C}, \ref{sec:Appendix_D} and \ref{sec:extra_FairTME_models} form the appendices and contain most of the proofs for the stated results. 

A few comments about notation; we use bold upper-case and lower-case letters for matrices and column vectors respectively. For $\bm{A}, \bB \in \R^{n\times n}$ we let $\langle \bm{A}, \bB \rangle := \text{Tr}(\bm{A}^\top \bB) = \text{Tr}(\bB^\top \bm{A})$ denote the trace inner product which induces the Frobenius norm, $\|\bm{A}\|_F := \sqrt{\langle \bm{A}, \bm{A} \rangle}$. The vectorization of a matrix is denoted by $\vect(\bm{A}) \in \R^{n^2}$ and is the vector formed by stacking the columns of $\bm{A}$. For $\bx \in \R^n$, $\Diag(\bm{x}) \in \R^{n\times n}$ is the diagonal matrix whose diagonal elements are the elements of $\bm{x}$; $\bm{I}_{d \times d} \in \R^{d \times d}$ denotes the identity matrix with the subscript suppressed when the dimension is clear from the context. We let $\Symn, \Symn_{+}, \Symn_{++}$ denote the set of $n\times n$ real symmetric, positive semidefinite and positive definite matrices respectively.  

\section{Obtaining Unconstrained Models via Matrix Functions}\label{sec:matrix_functions}
The focus of this section is defining matrix functions and demonstrating how they can be leveraged to rewrite the constrained optimization model \eqref{eqn:main_model} as an unconstrained model over a matrix vector space. Multiple generalizations of matrix functions exist in the literature \cite{bhatia2013matrix,hawkins1973generalized,H08,HJ, rinehart1955equivalence}
but for the purposes of this paper we focus on a simplified version of the definition provided in \cite{H08}. 

\begin{definition}\label{def:mat_func}
Let $F:\R \rightarrow \R$ be defined on the spectrum of $\bX \in \mathcal{S}^{n\times n}$ with spectral decomposition $\bX= \bQ \bLambda \bQ^\top$ with $\bLambda$ diagonal and $\bU$ orthogonal. Then, the matrix function defined by the scalar function $F$ is given by,
\begin{equation}\label{eqn:mat_func}
\bF(\bX) = \bQ \bF(\bLambda) \bQ^\top
\end{equation}
where $\bF(\bLambda):= \Diag(F(\Lambda_{11}), \hdots, F(\Lambda_{nn}))$. 
\end{definition}

As detailed in \cite{H08,HJ}, the value of $\bF(\bX)$ is independent of the spectral decomposition chosen for $\bX$ which removes the anxiety of operating with an ill-defined function. Additionally, and most crucially for the implementation of cubic-regularized Newton, we can compute the derivatives of matrix functions provided $F: \R \rightarrow \R$ is sufficiently differentiable. Higham (Chapter 3 of \cite{H08}) presents some results for the differential properties of matrix functions, and  Horn and Johnson (Section 6.6 in \cite{HJ}) establish a litany of general derivative formulas; however, their derivative computations become rather involved and computationally unappealing. It is useful to have the following derivative formulas from \cite{B05}. 

\begin{theorem}({\it Proposition 3.2 and Theorem 3.6 in \cite{B05}})\label{thm:primary_matrix_func_deriv}\; Let $F \in C^2(\R)$ and $\bX \in \mathcal{S}^{n\times n}$ with $\bX = \bQ \bLambda \bQ^\top$ where $\bLambda = \Diag(\lambda_1, \hdots, \lambda_n)$. Then, for all $\bH \in \Symn$ with $\bK := \bQ^\top \bH \bQ$,
\begin{align}
	DF(\bX)[\bH] &= \bQ \left( DF( \bLambda)[\bK] \right) \bQ^\top, \nonumber \\ \nonumber \\
	D^2F(\bX)[\bH,\bH] &= 2\bQ \left( D^2F(\bLambda)[\bK,\bK] \right) \bQ^\top, \nonumber 
\end{align}
where:
\[
DF(\bLambda)[\bK] = \sum_{i,j=1}^{n} F^{[1]}\left( \lambda_i, \lambda_j \right) \bE_{ii} \bK \bE_{jj},
\]
\[
D^2 F(\bLambda)[\bK,\bK] = \sum_{i,j,k=1}^{n} F^{[2]}\left(\lambda_i, \lambda_j,\lambda_k\right) \bE_{ii} \bK \bE_{jj} \bK \bE_{kk},
\]
with $\bE_{ii}$ having 1 at $(i,i)$ and 0 otherwise, and $F^{[\ell]}$ is the $\ell$-th order divided difference equation. 
\end{theorem}

\noindent The purpose of Theorem \ref{thm:primary_matrix_func_deriv} will become clear in the next subsection where we detail how to rewrite the constrained optimization model \eqref{eqn:main_model}. This result enables a closed form manner for writing the directional derivatives of the matrix functions which must be computed in-order to apply cubic-regularized Newton to solve the unconstrained model.

\subsection{Writing the Unconstrained Model} 
The role of matrix functions in this paper are to enforce the spectral constraints in \eqref{eqn:main_model}. All forms of interval constraints on the eigenvalues can be captured via matrix functions. Theorem \ref{thm:transform_problem} provides at least one possible method for rewriting \eqref{eqn:main_model} for various intervals $\mathcal{I}\subseteq \R$.

\begin{theorem}\label{thm:transform_problem}
Let $F:\Symn \rightarrow \R$ be in $C^2(\Symn)$ and assume there exists,
\[
\bX^* \in \argmin\left\{ F(\bX)\; | \; \lambda_j(\bX) \in \mathcal{I}_k,\; j=1,\hdots, n \;,\; \bX \in \Symn \right\},\]
then there exists, 
\[\bU^* \in \argmin\left\{ F(\bG_k(\bU))\; | \; \bU \in \Symn\right\},\]
such that $F(\bX^*) = F(\bG_k(\bU^*))$ where $\mathcal{I}_k$ and $\bG_k(\cdot)$ are given respectively by:
\begin{itemize}
\item $\mathcal{I}_1 = (\alpha, \;+\infty)$,\;$\bG_1(\bX) = e^{\bX} + \alpha \bm{I}$
\vspace{0.1in}
\item $\mathcal{I}_2 = (-\infty, \alpha)$,\;$\bG_2(\bX) = -e^{\bX} + \alpha \bm{I}$
\vspace{0.1in}
\item $\mathcal{I}_3 = [\alpha, \;+\infty)$,\; $\bG_3(\bX) = \bX^2 + \alpha \bm{I}$
\vspace{0.1in}
\item $\mathcal{I}_4 = (-\infty,\alpha]$,\; $\bG_4(\bX) = -\bX^2 + \alpha \bm{I}$
\vspace{0.1in}
\item $\mathcal{I}_5 = (\alpha, \beta]$,\; $\bG_5(\bX) = (\beta-\alpha) \left( \bm{I} + \bX^2\right)^{-1} + \alpha \bm{I}$ 
\vspace{0.1in}
\item $\mathcal{I}_6 = [\alpha, \beta)$ and $\bG_6(\bX) =  (\alpha-\beta) \left( \bm{I} + \bX^2\right)^{-1} + \beta \bm{I}$ 
\vspace{0.1in}
\item $\mathcal{I}_7 = (\alpha, \beta)$ and $\bG_7(\bX) = \bQ \bG_7(\bD) \bQ^\top$ with $G_7(x) = \left(\frac{\beta-\alpha}{\pi}\right) \text{arctan}(x) + \frac{\beta+\alpha}{2}$ 
\end{itemize}
and $\alpha, \beta \in \R$ with $\alpha < \beta$. 
\end{theorem}
\begin{proof}{
To prove this result, we only need to demonstrate that the mappings, 
\[
\bG_k: \Symn \rightarrow \left\{ \bX \in \Symn \;| \;  \lambda_i(\bX) \in \mathcal{I}_k,\; j=1,\hdots, n \right\}
\]
are smooth and surjective. Since all of the scalar functions which define $\bG_k$ are smooth it follows by the results of \cite{B05,H08,HJ} that $\bG_k$ is smooth. 
The surjectivity of the maps follow from the surjectivity of the scalar functions 
which define $\bG_k$ onto $\mathcal{I}_k$. 
For example, the scalar function defining $\bG_1$ is 
$G_1(x) = e^x + \alpha$ and it is smooth and surjective on $\mathcal{I}_1$.}  
\end{proof}

\section{Matrix Calculus and a Second-Order Chain Rule Identity}\label{sec:matrix_calculus}
Applying higher-order methods such as cubic-regularized Newton requires second-order derivative information. Thus, we must define gradient and Hessian operators as well as directional derivatives for real-valued matrix functions $F: \mathcal{V} \rightarrow \R$ where $\mathcal{V} \subseteq \mathbb{R}^{n \times n}$ is a matrix vector space. 

\begin{definition} The matrix gradient for $F:\mathcal{V} \rightarrow \R$ is a function $\bm{\nabla}F: \mathcal{V} \rightarrow \mathcal{V}$ whose elements are defined as, 
\[
\left( \bnabla F(\bX) \right)_{ij}:= \frac{\partial F}{\partial X_{ij}}(\bX), \;\;\; \forall \; i,j \in \{1,2\hdots, n\},
\]
provided all the partial derivatives exist. 
\end{definition}

\noindent We further define the directional derivative of a real-valued matrix function, and relate it to the matrix gradient in the expected manner.

\begin{definition}\label{def:dir_deriv_mat_func} The directional derivative of $F:\mathcal{V} \rightarrow \R$ at $\bX \in \mathcal{V}$ in the direction of $ \bH\in \mathcal{V}$, denoted $DF(\bX)[\bH]$, is defined as, 
\begin{equation}\label{eqn:dir_deriv_dt}
DF(\bX)[\bH] := \frac{d}{dt} \left( F(\bX + t\bH) \right) \big|_{t=0}.
\end{equation}
\end{definition}

\noindent As in the Euclidean setting, the directional derivatives of a real-valued matrix function relates closely to its gradient, 
\begin{equation}\label{eqn:dir_deriv_grad}
DF(\bX)[\bH] = \langle \bm{\nabla} F(\bX), \bH \rangle.
\end{equation}
Note, Definition \ref{def:dir_deriv_mat_func} could have been replaced with the more general Fr\'echet derivative of $F$ where ${DF}(\bX): \mathcal{V}\rightarrow \R$ is a linear operator and,  
\[
\lim_{\bH \rightarrow \bm{0}} \frac{|F(\bX+\bH) - F(\bX) - DF(\bX)[\bH]|}{\|\bH\|_F}  = 0. 
\] 
Moving to higher-order derivatives, we define the Hessian of a real-valued matrix function. 
\begin{definition}The Hessian of $F:\mathcal{V} \rightarrow \R$ at $\bX \in \mathcal{V}$ is a function $\bm{\nabla^2} F(\bX)[\cdot]: \mathcal{V} \rightarrow \mathcal{V}$ where, 
\begin{equation}\label{eqn:mat_hess}
\bm{\nabla^2} F(\bX)[\bH] := \frac{d}{dt} \left( \bm{\nabla} F(\bX + t\bH) \right) \big|_{t=0}.
\end{equation}
\end{definition}
\noindent The second-order directional derivative of $F$ at $\bX \in \mathcal{V}$, $D^2 F(\bX)[\cdot, \cdot] : \mathcal{V} \times \mathcal{V} \rightarrow \R$, can then be defined through the Hessian of $F$ at $\bX$ as, 
\begin{equation}\label{eqn:2nd_order_DD}
D^2 F(\bX)[\bH_1, \bH_2] := \langle \bm{\nabla^2 F}(\bX)[\bH_1], \bH_2 \rangle.
\end{equation}
It is important to note $\bm{\nabla^2} F(\bX)$ is uniquely defined and $D^2 F(\bX)$ is a bi-linear, symmetric operator in $\bH_1$ and $\bH_2$. Using this differential information one can present a matrix version of Taylor's Theorem, 
\[
F(\bX + \bH ) = F(\bX) + DF(\bX)[\bH] + \frac{1}{2}D^2 F(\bX)[\bH,\bH] + \mathcal{O}(\|\bH\|_F^3).
\]
Additionally, we can define an operator norm on the second-order matrix derivative of $F$. 
\begin{definition}
The operator norm of $D^2 F(\bX)$ over the matrix vector space $\cV$ is given as, 
\[
\|D^2 F(\bX)\| := \max\left\{ | D^2 F(\bX)[\bH_1, \bH_2]| \; \big| \; \|\bH_1\|_F = \|\bH_2\|_F = 1, \; \bH_1, \bH_2 \in \mathcal{V} \right\}.
\]
\end{definition}

\noindent Using this norm, we introduce the notion of a Lipschitz Hessian which plays a crucial role in the convergence analysis of the cubic-regularized Newton algorithm.  

\begin{definition}
The function $F: \mathcal{V} \rightarrow \R$ has Lipschitz Hessian with parameter $L>0$ if,
\[
\| D^2 F(\bX) - D^2 F(\bY) \| \leq L \|\bX - \bY \|_F,\; \;\; \forall \; \bX, \bY \in \mathcal{V}. 
\]
\end{definition} 
A couple of key inequalities follow from assuming a Lipschitz Hessian. Lemma \ref{lem:hess_lip}, which is proved in the Supplemental Materials, is crucial in Appendix \ref{sec:Appendix_D}.
\begin{lemma}\label{lem:hess_lip}
If the Hessian of $F: \mathcal{V} \rightarrow \R$ is Lipschitz continuous on $\mathcal{V}$ with parameter $L>0$, then for all $\bX,\bY \in \mathcal{V}$, 
\begin{equation}\label{eqn:lip_ineq_1}
\| \bm{\nabla} F(\bX) - \bm{\nabla} F(\bY) - \bm{\nabla^2} F(\bX)[\bY-\bX] \| _F \leq \frac{L}{2} \| \bX - \bY \|_F^2,
\end{equation}
\begin{equation}\label{eqn:lip_ineq_2}
| F(\bY) - F(\bX)- \langle \bm{\nabla} F(\bX), \bY-\bX \rangle  - \frac{1}{2}\langle \bm{\nabla^2} F(\bX)[\bY-\bX], \bY-\bX \rangle| \leq \frac{L}{6} \| \bX - \bY \|_F^3.
\end{equation}
\end{lemma}

Since $\bm{\nabla ^2} F(\bX)[\bH]$ is a linear function in $\bH$, it follows there exists $\bm{A}(\bX) \in \R^{n^2 \times n^2}$ such that,
\begin{equation}\label{eqn:hess_vec}
\text{vec}\left( \bm{\nabla ^2} F(\bX)[\bH] \right) = \bm{A}(\bm{X})\vect(\bH).
\end{equation}
We will subsequently refer to $\bm{A}(\bX)$ as the vectorization of the Hessian, and \eqref{eqn:hess_vec} will play a fundamental role in our new analysis for cubic-regularized Newton in Section \ref{sec:Appendix_D}. For more details on matrix derivatives and directional derivatives, the reader is directed to \cite{A09,HJ,Lee}.

A final important result for implementing cubic-regularized Newton is the chain rule for the composition of functions which appears in \eqref{eqn:unc_main_model}. 

\begin{theorem}\label{thm:CR}
Let $\mathcal{V} \subseteq \R^{n \times n}$ be a vector space, $F: \mathcal{V} \rightarrow \R$ and $\bG: \mathcal{V} \rightarrow \mathcal{V}$ be any matrix function defined by an analytic scalar function. Let $\bX, \bH \in \mathcal{V}$ and assume $F$ is differentiable at $\bG(\bX)$ and $\bG$ is differentiable at $\bX$. Then, assuming $F$ and $\bG$ are two-times continuously differentiable, the first and second-order directional derivatives of $F \circ \bG$ at $\bX$ in the direction of $\bH$ are,
\begin{equation}\label{eqn:CR_1st}
D(F\circ \bG)(\bX)[\bH] = DF(\bG(\bX))\bigg[ D\bG(\bX)[\bH] \bigg],
\end{equation} 
\begin{multline}\label{eqn:CR_2nd}
D^2(F\circ \bG)(\bX)[\bH,\bH] = \\ D^2F(\bG(\bX))\bigg[ D\bG(\bX)[\bH], D\bG(\bX)[\bH]\bigg] + DF(\bG(\bX))\bigg[D^2 \bG(\bX)[\bH,\bH]\bigg].
\end{multline} 
\end{theorem}
It is worth stating \eqref{eqn:CR_1st} holds true for general vector spaces, and this identity is a standard result (see e.g., \cite{A09,Lee}). On the other hand, we were unable to locate \eqref{eqn:CR_2nd} in the literature with the general setting of matrix functions. We thus prove it in Section \ref{sec:Appendix_C}. 

\section{Cubic-Regularized Newton for Matrix Optimization}\label{sec:CRN_for_matrix_opt}
We introduce the cubic regularized Newton method for the unconstrained matrix optimization problem,
\vspace{-0.1in}
\begin{align}\label{eqn:crn_prob}
\min&\;\; F(\bX) \\
\text{s.t.}&\;\; \bX \in \mathcal{V}, \nonumber 
\end{align}
where we assume: $\mathcal{V} \subset \R^{n\times n}$ is a $d$-dimensional subspace, $F \in C^2(\mathcal{V})$ has Lipschitz Hessian with parameter $L>0$, and there exists finite $F^*$ such that $F^* \leq F(\bX)$ for all $\bX \in \mathcal{V}$. 

In our discussion of convergence, we restrict ourselves to $\epsilon$-stationary points due to the potential non-convexity of \eqref{eqn:crn_prob} which prohibits us from claiming global optimal solutions to our general matrix optimization model. The point  $\bX^* \in \mathcal{V}$ is a second-order $\epsilon$-stationary point of \eqref{eqn:crn_prob} with $\epsilon > 0$ if $\|\bnabla F(\bX^*)\|_F \leq \epsilon$,
and $\langle \bnabla^2 F(\bX^*)[\bH], \bH \rangle \geq -\sqrt{\epsilon} \|\bH\|_F^2$\; $\forall \bH \in\mathcal{V}$.

Algorithm \ref{alg:CRN} sketches the cubic-regularized Newton method. This method has two key attributes, which are encapsulated in the theorems below: (1) global convergence to second-order $\epsilon$-stationary points and (2) local superlinear convergence near non-degenerate local minimum. 

\begin{theorem}\label{thm:global_convergence_CRN} Assume $F \in C^2(\mathcal{V})$ has Lipschitz Hessian with parameter $L>0$ and is bounded below by $F^*$ on $\mathcal{V}$. Then the sequence $\{\bX_k\}$ generated by Algorithm \ref{alg:CRN} converges to a second-order $\epsilon$-stationary point for \eqref{eqn:crn_prob} with rate $\mathcal{O}(\epsilon^{-3/2})$. 
\end{theorem}

\begin{theorem}\label{thm:superlinear_CRN} If Algorithm \ref{alg:CRN} is initialized sufficiently close to a non-degenerate local minimum $\bX^* \in \mathcal{V}$, i.e. 
\[
\bnabla F(\bX^*) = 0 \;\; \text{ and } \;\; \exists \delta > 0 \; \text{ such that } \;\; \langle \bnabla^2 F(\bX^*)[\bH], \bH \rangle \geq \delta \|\bH\|_F^2 \;\; \text{for all } \bH \in \mathcal{V},
\]
then the sequence $\{\bX_k\}$ generated by Algorithm \ref{alg:CRN} will convergence quadratically to $\bX^*$. 
\end{theorem}

\begin{algorithm}[!htbp]
\caption{Cubic Regularized Newton for \eqref{eqn:crn_prob}}\label{alg:CRN}
\begin{algorithmic}
\State{\bf Input:} Choose $\bX_0 \in \cV$ and $L_0 \in (0, L]$.\vspace{0.05in}
\For{$k=0,1,2\hdots$} \vspace{0.05in}
	\State(1) Find $M_k \in [L_0, 2L]$ such that, 
	\[ 
		F(\bm{T}_{M_k}(\bX_k)) \leq F(\bX_k), 
	\]
	\vspace{-0.3in}
	 \State where $\bm{T}_{M}(\bX_k)$ is an element of the Argmin set,
	\vspace{-0.1in}
	 \begin{equation}
	\underset{{\bY \in \mathcal{V} }}{\text{Argmin}}\left[ \langle \bm{\nabla} F(\bX_k), \bY-\bX_k \rangle + \frac{1}{2} \langle \bm{\nabla^2} F(\bX_k)[\bY-\bX_k], \bY-\bX_k \rangle + \frac{M_k}{6} \|\bY-\bX_k\|_F^3 \right] \nonumber 
	\end{equation}
	\State(2)\;\; $\bX_{k+1} = \bm{T}_{M_k}(\bX_k)$ \vspace{0.05in}
\EndFor
\end{algorithmic}
\end{algorithm}

Theorems \ref{thm:global_convergence_CRN} and \ref{thm:superlinear_CRN} were first proven in \cite{Nes06} over $\R^n$. Nesterov then extended the convergence analysis over general vector spaces with convex constraints and convex functions in \cite{nesterov2006cubic}. Theorem \ref{thm:global_convergence_CRN} and a superlinear convergence result was proven over Riemannian manifolds in \cite{Jun18} for a variant of cubic-regularized Newton. The contribution of this paper in terms of the analysis of the algorithm comes in the form of the proof technique. The proofs in Section \ref{sec:Appendix_D} for these results rely upon vectorizing the matrix subproblem and proving lemmas which relate the vectorized problem to the matrix problem. Upon connecting these two models, the original analysis in \cite{Nes06} can be applied to prove the results in the more general matrix setting.

\section{Implementation of Cubic-Regularized Newton for Matrix Optimization}\label{sec:implement}
We now discuss implementation of Algorithm \ref{alg:CRN}. The key hurdle to applying the method rests primarily on solving a matrix optimization subproblem of the form, 
\begin{equation}\label{eqn:mat_subprob}
\min_{\bH \in \mathcal{V}}\;\;\langle \bG, \bH \rangle + \frac{1}{2} \langle \mathcal{L}(\bH), \bH \rangle + \frac{M}{6} \|\bH\|_F^3,
\end{equation}
where $\mathcal{L}: \mathcal{V} \rightarrow \mathcal{V}$ is a linear map, $\bG\in \mathcal{V}$ and $M>0$. In this section we describe one approach to obtain a global optimal solution to solve \eqref{eqn:mat_subprob} by first rewriting it as an unconstrained cubic-regularized subproblem over $\R^d$. First, select an orthonormal basis, $\{\bE_1, \hdots, \bE_d\}$,  for the d-dimensional vector subspace $\mathcal{V} \subseteq \R^{n\times n}$. Then, any $\bH\in \mathcal{V}$ can be rewritten in terms of the basis as, 
\[
\bH = \sum_{i=1}^{d} z_i \bE_i,  
\]
where $\bz \in \R^d$. Using the decomposition of $\bH$ we then rewrite the terms in the objective function of the model: 
\begin{equation}\label{eqn:b_vec_mat}
\langle \bG, \bH \rangle = \sum_{i=1}^{d} z_i \langle \bG, \bE_i \rangle, 
\end{equation}
\begin{equation}
    \langle \mathcal{L}(\bH), \bH \rangle = \langle \sum_{i=1}^{d} z_i \mathcal{L}(\bE_i), \sum_{j=1}^{d} z_j \bE_j  \rangle = \sum_{i=1}^{d} \sum_{j=1}^{d} z_i z_j \langle \mathcal{L}(\bE_i), \bE_j \rangle, 
\end{equation}
and because the $\bE_i$'s are orthonormal it follows $\| \bH \|_F^3 = \|\bz\|^3$; therefore, through the basis, we can compute the optimal solution to \eqref{eqn:mat_subprob} by solving the unconstrained vector optimization model, 
\begin{equation}\label{eqn:vec_subprob_main}
\min_{\bz \in \R^d}\;\; \bm{b}^\top \bz + \frac{1}{2}\bz^\top \bm{A} \bz + \frac{M}{6}\|\bz\|^3,
\end{equation}
where $b_i = \langle \bG, \bE_i\rangle$ and $A_{ij} = \langle \mathcal{L}(\bE_i), \bE_j \rangle$ for $i,j = 1,\hdots, d$. 

Many efficient solution methods exist to solve the vectorized subproblem \cite{C19,jiang2021accelerated,Nes06}. Therefore, this approach for solving the matrix subproblem will be tractable provided the cost of constructing the vectorized subproblem is reasonable. From Algorithm \ref{alg:CRN}, we see $\bG = \bm{\nabla} F(\bX_k)$ and 
$\mathcal{L}(\bH):= \bm{\nabla^2}F(\bX_k)[\bH]$. 
Thus, in-order to form the subproblem, it will cost one gradient computation and $d$-Hessian evaluations of the objective function where the Hessians will be computed at the basis matrices. This means the overall cost of forming the vectorized subproblem will depend on both the complexity of the objective function's derivatives and the complexity of the basis for the matrix vector space.  

Another approach to solving this model would be to apply a Lanczos-based approach as described by Agarwal et al. in Section 8 of \cite{Ag21}. This is a basis-free approach, but it would still require multiple calls to the Hessian and potentially $d$ such calls if the global optimal solution is desired for \eqref{eqn:mat_subprob}. Though an inexact version of cubic-regularized Newton is not implemented in this paper such methods exist and would benefit computationally from a Lanczo-based approach. However, there is still much room for improvement on producing efficient methods for solving the matrix subproblem, and the recent works on \eqref{eqn:vec_subprob_main} might enable improvements in the matrix setting.     

\section{Application to Fairness}\label{sec:experiments}

We present a new model for fair and robust covariance matrix estimation which we solve by our proposed procedure. More precisely, we estimate the shape or scatter matrix, which is a scalar multiple of the covariance matrix when the covariance exists (for simplicity, we refer to the covariance).
Our fairness model is an adapted variant of the Tyler's M-estimator (TME) model; we call it the Fair TME model. The Fair TME model seeks to simultaneously obtain both a robust and fair covariance matrix for multiple subgroups of a single dataset. Substantial effort has been done on robust covariance matrix estimation; however, it appears no one has considered fairness in their estimations. Our fairness model takes a single statistic for the subgroups of a dataset, which we derive from the TME model, and through a minimization problem seeks to make this statistic small and similar amongst the subgroups. This concept is exceptionally general and therefore extends to other data science problems. 

Section \ref{sec:TME_FairTME} reviews the TME model and presents our proposed Fair TME model. Sections \ref{sec:num_exp_sim} and \ref{sec:num_exp_real} provide the results of the numerical experiments we performed comparing the two models on simulated and real datasets respectively. 

\subsection{Review of the TME Model and Introduction to the Fair TME Model}\label{sec:TME_FairTME}
The Tyler's M-estimator (TME) model proposed by Tyler in \cite{T87} computes a robust estimator for the covariance matrix of a centered dataset. It is determined by solving the following equation: 
\begin{equation}\label{eqn:TME_eq}
\textbf{Find}\;\; \bR \in \mathcal{S}^{p \times p}_{++} \;\; \textbf{s.t.}\;\; \frac{p}{n} \sum_{i=1}^{n} \frac{\bx_i \bx_i^\top}{\bx_i^\top \bR^{-1} \bx_i} -  \bR = 0.  
\end{equation} 
Since \eqref{eqn:TME_eq} is invariant to scaling in the matrix $\bR$, which makes sense for shape matrix estimation, a trace constraint is often placed on $\bR$, i.e. $\trace(\bR) = p$. It has been shown the solution to \eqref{eqn:TME_eq} can be found by solving the following minimization problem, 
\begin{align}\label{eqn:TMEorig}
\min&\;\; \frac{p}{n} \sum_{i=1}^{n} \log(\bx_i^\top \bR^{-1} \bx_i) + \log \det \bR \\
\text{s.t.}&\;\; \bR \in \mathcal{S}^{p\times p}_{++}.\nonumber 
\end{align}
The relationship between \eqref{eqn:TME_eq} and \eqref{eqn:TMEorig} is clear from the first-order optimality condition for \eqref{eqn:TMEorig}. Additionally, efficient methods exist to solve for TME. Tyler \cite{T87} presented a fixed-point method which solves \eqref{eqn:TMEorig} under various data settings; this approach has been investigated extensively and has proven linear convergence under different data assumptions \cite{F20,G20}. 

We propose a Fair TME model which seeks a single robust covariance matrix that achieves a level of fairness between multiple subgroups of a dataset. Assume we have $r$ independent centered subgroups $\left\{ \{\bx^1_i\}_{i=1}^{n_1}, \hdots, \{\bx^r_i\}_{i=1}^{n_r} \right\}$ contained in $\R^p$ which could vary in size and be drawn from multiple distributions. Define the functions, 
\begin{equation}\label{eqn:fj_defn}
f_j(\bR) := \frac{p}{n_j} \sum_{i=1}^{n_j} \log((\bx_i^j)^\top \bR^{-1} (\bx_i^j)) + \log \det \bR, \;\; j=1, \hdots, r.
\end{equation}
Letting  $f_j^*:= \min\left\{ f_j(\bR)\; | \; \bR \in \mathcal{S}^{p \times p}_{++} \right\}$, we define the {\it TME error} for the $j$-th subgroup to be,
\begin{equation}\label{eqn:residual}
 \cE_j(\bR):= f_j(\bR) - f_j^*, \;\; j=1, \hdots, r. 
\end{equation}
We then define a perfectly fair covariance matrix for multiple subgroups of a dataset. 
\begin{definition}\label{def:fair_covmat}
The matrix $\bR_f \in \mathcal{S}^{p \times p}_{++}$ is a perfectly fair covariance matrix for the subgroups $\left\{ \{\bx^1_i\}_{i=1}^{n_1}, \hdots, \{\bx^r_i\}_{i=1}^{n_r} \right\}$ contained in $\R^p$ provided, 
\[
\cE_i(\bR_f) = \cE_j(\bR_f), \;\; \forall \; i,j \in \{1,\hdots, r\}. 
\]
\end{definition}
\noindent A perfectly fair covariance matrix provides equality in the TME errors. It is often unachievable, so we propose an optimization model which seeks a nearly perfectly fair covariance while minimizing the individual TME errors. Thus, our Fair TME model is stated as follows:   
\begin{align}\label{eqn:fair_TME}
\min&\;\;\; \mu_1 \sum_{j=1}^{r} \cE_j(\bR) + \frac{\mu_2}{2}\sum_{i=1}^{r-1}\; \sum_{j=i+1}^{r}\left( \cE_i(\bR) - \cE_j(\bR) \right)^2 \\
\text{s.t.}&\;\; \bR \in \mathcal{S}^{p\times p}_{++},\nonumber 
\end{align}
where $\mu_1, \mu_2 \geq 0$ are constants. These constants dictate the focus of the model. If $\mu_2 >> \mu_1$, then fairness is given priority over the subgroups having small TME errors. This means a perfectly fair covariance might be obtained but at the cost of large $\cE_j$ values. The most desirable situation is obtaining a solution $\bR^*$ to \eqref{eqn:fair_TME} giving an objective value of zero. This would be finding a perfectly fair and universally robust covariance matrix. Since such a solution often does not exist, we seek a local minimizer of \eqref{eqn:fair_TME}. 

The non-convexity of the model and the fact it is a constrained problem over the set of positive definite matrices makes it a candidate for the methods outlined in this paper. One possible route to solve the model is to apply the change of variable $\bZ = \bR^{-1}$ to \eqref{eqn:fair_TME} and then utilize the matrix function $\bG(\bR)=\bR^2$. These steps produce the unconstrained model,  
\begin{align}\label{eqn:fair_TME_uncon_sq}
\min&\;\;\; \mu_1 \sum_{j=1}^{r} \hat{\cE}_j(\bR^2) + \frac{\mu_2}{2}\sum_{i=1}^{r-1}\; \sum_{j=i+1}^{r} \left( \hat{\cE}_i(\bR^2) - \hat{\cE}_j(\bR^2) \right)^2 \\
\text{s.t.}&\;\; \bR \in \mathcal{S}^{p\times p},\nonumber 
\end{align}
where,
\begin{equation}\label{eqn:Ehat_defn}
\hat{\cE}_j(\bR) := \frac{p}{n_j} \sum_{i=1}^{n_j} \log((\bx_i^j)^\top \bR (\bx_i^j)) - \log \det \bR - f_j^*, \;\; j=1, \hdots, r.
\end{equation}
Thus, if $\bR^*$ solves \eqref{eqn:fair_TME_uncon_sq}, then $(\bR^*)^{-2}$ solves \eqref{eqn:fair_TME}. Theorem \ref{thm:transform_problem} guarantees and enables us to obtain a solution to \eqref{eqn:fair_TME} from \eqref{eqn:fair_TME_uncon_sq} while Theorems \ref{thm:primary_matrix_func_deriv} and \ref{thm:CR} allow us to calculate the derivatives needed to apply Algorithm \ref{alg:CRN} to solve \eqref{eqn:fair_TME_uncon_sq}. While first-order methods could also be applied to \eqref{eqn:fair_TME_uncon_sq}, they will often lack the benefits of local superlinear convergence and guarantees to return second-order stationary points (see discussion of Figure \ref{fig:converg_plot}).

Other approaches to solving \eqref{eqn:fair_TME} can be implemented besides solving \eqref{eqn:fair_TME_uncon_sq}. We present one other method in Section \ref{sec:extra_FairTME_models} and compare it to the approach above. 


\subsection{Fair TME vs. Standard TME Numerical Experiments: Simulated Data}\label{sec:num_exp_sim}

For our numerical experiment with simulated data we generated a dataset containing four centered elliptical subgroups with varying numbers of points in $\R^{30}$. The points in the subgroups were generated from the expression, 
$
\bx = \bm{S_p}\bm{\xi},
$
where $\bm{S_p} \in \mathcal{S}^{30 \times 30}_{++}$ and $\bm{\xi}$ was drawn uniformly at random from the sphere in $\R^{30}$. The matrices $\bm{S_p}$ were formed in MATLAB 2021b as, 
\[
\bA = randn(30,30);\; \bm{S_p} = \bA\bA^\top + \delta\bm{I},
\]
where $\delta$ was set to $10^{-8}$ to ensure positive definiteness. We fixed the random number seeds to be {\it rng(1)}, {\it rng(5)}, {\it rng(7)} and {\it rng(11)} to generate subgroups containing 50, 100, 200 and 75 points respectively.   

We computed stationary points to the Fair TME model \eqref{eqn:fair_TME} in all of our experiments by solving \eqref{eqn:fair_TME_uncon_sq} with cubic-regularized Newton. To compare the Fair TME model to the standard TME model \eqref{eqn:TMEorig} we solved the TME model for the entire dataset using the standard fixed-point approach. All numerical experiments we conducted in MATLAB 2021b. 

For each experiment we provide a table presenting two key quantities: the TME errors \eqref{eqn:residual} and the fairness value. The fairness value is defined as the maximum absolute difference between the TME errors for the individual subgroups at the solution matrix $\bR^*$, i.e., 
\begin{equation}\label{eqn:fairness_value}
\textbf{Fairness Value} := \max\left\{ | \cE_i(\bR^*) - \cE_j(\bR^*)| \; \big| \; i,j \in \{1,\hdots, r\} \right\}. 
\end{equation}
A perfectly fair covariance matrix will have a fairness value of zero. Ideally a fair and robust covariance matrix will have low TME errors for the subgroups and a low fairness value. 

\begin{table}[!htbp]
\centering
\begin{tabular}{cllll}
\hline
\multicolumn{5}{c}{\bf Fair TME} \\ \hline
$(\mu_1, \mu_2)$ & \multicolumn{2}{c}{TME Errors} & \multicolumn{2}{c}{Fairness Value} \\ \hline 
$(1,1)$  & \multicolumn{2}{c}{[ 81.2524,  81.2882,  81.2671,  81.2483]}           & \multicolumn{2}{c}{   0.03993}               \\ 
$(5,1)$  & \multicolumn{2}{c}{[ 81.2040,  81.3778,  81.2751,  81.1842]}           & \multicolumn{2}{c}{   0.19357}               \\ 
$(1,5)$  & \multicolumn{2}{c}{[ 81.2625,  81.2696,  81.2654,  81.2616]}           & \multicolumn{2}{c}{   0.00804}               \\ 
$(10,1)$ & \multicolumn{2}{c}{[ {\bf81.1470},  81.4831,  81.2841,  81.1101]}           & \multicolumn{2}{c}{   0.37294}              \\ 
$(1,10)$ & \multicolumn{2}{c}{[ 81.2637,  {\bf81.2673},  {\bf81.2652},  81.2633]}           & \multicolumn{2}{c}{\bf0.00402}       \\ \hline
\multicolumn{5}{c}{\bf TME} \\ \hline 
& \multicolumn{2}{c}{TME Errors} & \multicolumn{2}{c}{Fairness Value} \\ \hline 
& \multicolumn{2}{c}{[108.0067,  97.6193,  86.6591,  {\bf54.2080}]} & \multicolumn{2}{c}{  53.79874}                \\ \hline 
\end{tabular}
\caption{\label{tab:simdat}Results of the simulated data example for varying choices of the parameters $\mu_1$ and $\mu_2$ on four synthetic elliptical subgroups in $\R^{30}$ containing 50, 75, 100 and 200 points respectively.
}
\end{table}
Table \ref{tab:simdat}  displays the results for the simulated data. From Table \ref{tab:simdat} we see the Fair TME model was superior in finding a fair covariance matrix for the subgroups while solving the original TME formulation with the entire dataset did not provide an equitable covariance matrix. For each parameter setting of $\mu_1$ and $\mu_2$ the fairness value obtained by solving the Fair TME model with cubic-regularized Newton via matrix functions outperformed the TME model. The Fair TME solutions generated fairness values which were at least 100 times smaller and up to 10,000 times smaller than the fairness value produced by the TME model.

Furthermore, each parameter setting of the Fair TME model produced better TME errors for all but the last subgroup when compared to those computed by the TME model. The TME errors for the TME model also suggest a level of bias. Indeed, the TME error for the subgroup with 50 points was about 108 while the TME error for the subgroup with 200 points was approximately 54. This large discrepancy seems to suggest the standard TME model gives preferential treatment to the subgroup contributing the most information to the dataset.

\begin{figure}[!htbp]
\centering
\includegraphics[width=\textwidth]{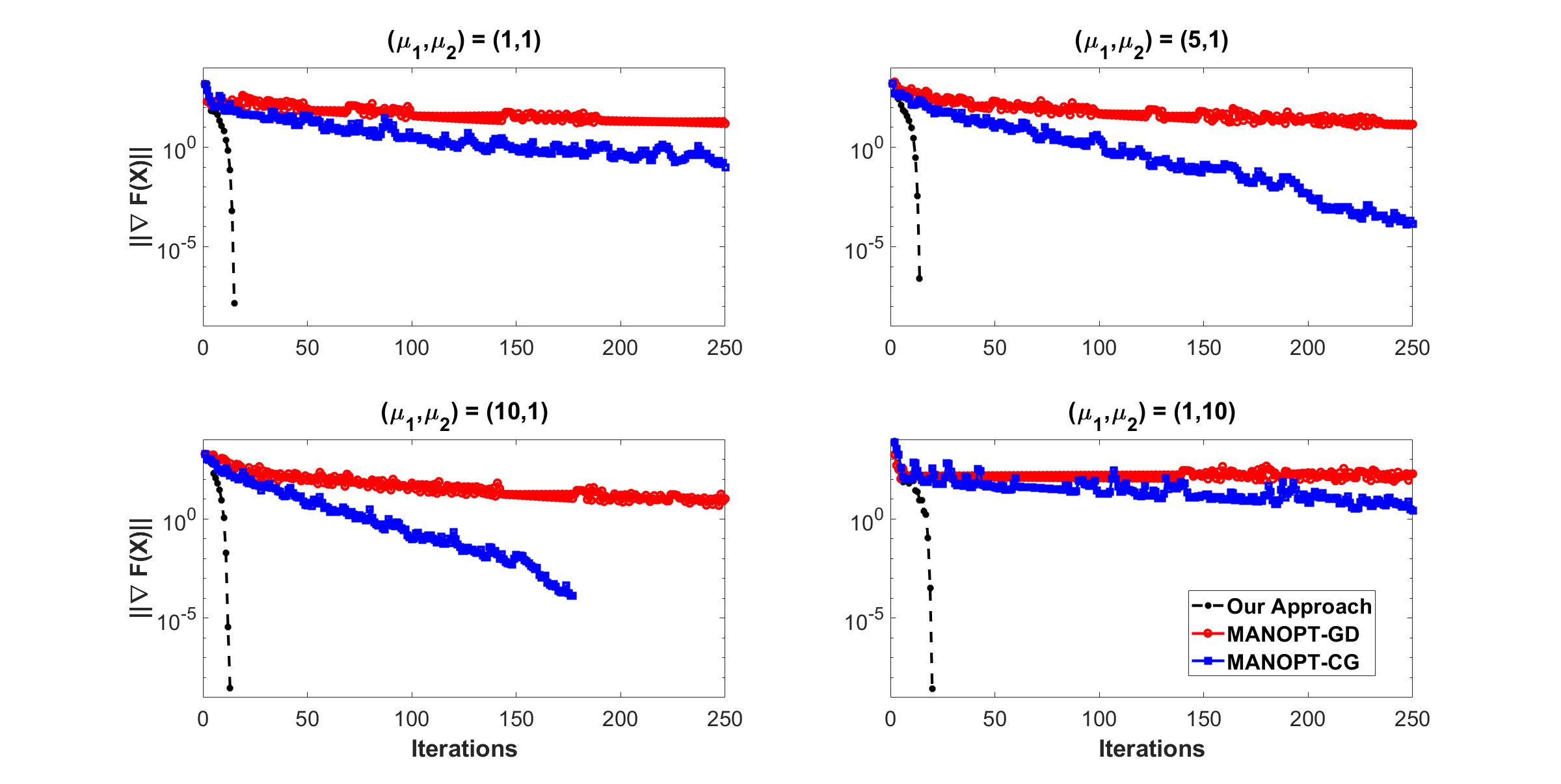}
\caption{\label{fig:converg_plot} Display of the convergence plots for four of the five instances of the Fair TME model in Table \ref{tab:simdat}. The title of each figure gives the parameter setting it corresponds to in Table \ref{tab:simdat}. The y-axis is in log-scale and provides the norm of the gradient at each iteration of the algorithm and the x-axis states the iteration count}
\end{figure}

For the Fair TME models solved in Table \ref{tab:simdat}, cubic-regularized Newton quickly solved \eqref{eqn:fair_TME_uncon_sq} with local superlinear convergence in each instance to second-order stationary points. Figure \ref{fig:converg_plot} displays a sample of the convergence plots for the models solved in the simulated experiment. The y-axis is in log-scale and provides the Frobenius norm of the gradient of \eqref{eqn:fair_TME_uncon_sq} at each iteration while the x-axis gives the iteration count. We also applied Manopt's \cite{MANOPT} gradient descent (GD) and conjugate gradient (CG) solvers on the model. Figure \ref{fig:converg_plot} shows both Manopt's GD and CG solvers failed to return a first-order $\epsilon$-stationary point with $\epsilon \leq 10^{-7}$ after 250 iterations while Algorithm \ref{alg:CRN} obtained a second-order  $\epsilon$-stationary point after less than 25 iterations. Actually, Manopt's GD and CG solvers failed to obtain an $\epsilon$-stationary-point with $\epsilon \leq 10^{-7}$ even after 2000 iterations.

\subsection{Fair TME vs. Standard TME Numerical Experiments: Real Data}\label{sec:num_exp_real}
We also conducted three experiments using real datasets from the UCI Machine Learning Repository \cite{UCIMLR}, namely, the Wine \cite{WineData}, Skills \cite{SkillData} and Credit \cite{yeh2009comparisons} datasets. Each real dataset tested was utilized recently in works on Fair PCA \cite{olfat2019convex,vu2022distributionally,yeh2009comparisons}. We processed these datasets according to the suggested procedure of \cite{olfat2019convex}. For the Fair TME model experiments, we removed categorical features, centered the subgroups of each dataset by subtracting their sample mean and then normalized all features to have unit variance. For the TME model experiments, the centering was performed for the whole dataset instead of the different subgroups. 

Table \ref{tab:datasets} provides the names of the various subgroups of each dataset and their respective dimensions. For each 
parameter setting tested in the three experiments, Algorithm \ref{alg:CRN} returned a second-order $\epsilon$-stationary point to \eqref{eqn:fair_TME_uncon_sq} with $\epsilon \leq 10^{-6}$ which we used to solve \eqref{eqn:fair_TME}.
\begin{table}[!htbp]
\centering 
\begin{tabular}{|cccc|}
\hline
\multicolumn{2}{|c|}{\textbf{Wine Datasets \cite{WineData}}}                            & \multicolumn{2}{c|}{\textbf{Skills Datasets \cite{SkillData}}}       \\ \hline
\multicolumn{1}{|c|}{Name}              & \multicolumn{1}{c|}{Dimension} & \multicolumn{1}{c|}{Name}                & Dimension \\ \hline
\multicolumn{1}{|c|}{\it Red\textunderscore good}          & \multicolumn{1}{c|}{855 x 11}  & \multicolumn{1}{c|}{\it Skill\textunderscore Tier\textunderscore 1}        & 514 x 15  \\
\multicolumn{1}{|c|}{\it Red\textunderscore bad}           & \multicolumn{1}{c|}{744 x 11}  & \multicolumn{1}{c|}{\it Skill\textunderscore Tier \textunderscore 2}        & 1364 x 15 \\
\multicolumn{1}{|c|}{\it White\textunderscore good}        & \multicolumn{1}{c|}{3258 x 11} & \multicolumn{1}{c|}{\it Skill\textunderscore Tier\textunderscore 3}        & 1427 x 15 \\
\multicolumn{1}{|c|}{\it White\textunderscore bad}         & \multicolumn{1}{c|}{1640 x 11}  & \multicolumn{1}{c|}{\it Skill\textunderscore Tier\textunderscore 4}        & 90 x 15   \\ \hline
\multicolumn{4}{|c|}{\textbf{Credit Datasets \cite{yeh2009comparisons} }}                                                                                 \\ \hline
\multicolumn{1}{|c|}{Name}              & \multicolumn{1}{c|}{Dimension} & \multicolumn{1}{c|}{Name}                & Dimension \\ \hline
\multicolumn{1}{|c|}{\it Male\textunderscore Single\textunderscore HgEd}  & \multicolumn{1}{c|}{5579 x 19} & \multicolumn{1}{c|}{\it Female\textunderscore Single\textunderscore HgEd}  & 8260 x 19 \\
\multicolumn{1}{|c|}{\it Male\textunderscore Single\textunderscore LwEd}  & \multicolumn{1}{c|}{974 x 19}  & \multicolumn{1}{c|}{\it Female\textunderscore Single\textunderscore LwEd}  & 1151 x 19 \\
\multicolumn{1}{|c|}{\it Male\textunderscore Married\textunderscore HgEd} & \multicolumn{1}{c|}{4062 x 19} & \multicolumn{1}{c|}{\it Female\textunderscore Married\textunderscore HgEd} & 6506 x 19 \\
\multicolumn{1}{|c|}{\it Male\textunderscore Married\textunderscore LwEd} & \multicolumn{1}{c|}{1128 x 19} & \multicolumn{1}{c|}{\it Female\textunderscore Married\textunderscore LwEd} & 1963 x 19 \\ \hline
\end{tabular}
\caption{\label{tab:datasets}The Wine and Skills datasets were partitioned into four subgroups and the Credit dataset was partitioned into eight subgroups. The first number in the dimension gives the number of samples; the second number provides the dimension of the samples}
\end{table}
\subsubsection{Wine Dataset Experiment} 

The Wine Dataset \cite{WineData} contains chemical composition information for 6497 red and white wines comprised of 11 input features and a single output feature being wine quality. The wine quality is  given as a score between 0 and 10. We partitioned this dataset into four subgroups based on the wine type (red or white) and wine quality. For instance, {\it Red\textunderscore good} is the set of red wines in the dataset which were given a quality rating between 6 and 10, and {\it White\textunderscore bad} is the set of white wines with a quality rating from 0 to 5. Table \ref{tab:wine} displays the results of the Fair and standard TME models on this dataset. 

From Table \ref{tab:wine} we see the Fair TME models performed the best in terms of both the TME errors and the fairness value. For each parameter setting the individual TME errors for the respective subgroups were smaller for the Fair TME model than for the standard TME model applied to the entire dataset. For the parameter setting $(\mu_1, \mu_2) = (1,10)$, the fairness value obtained by the Fair TME model was about 40 times smaller than the fairness value computed by the standard TME model. 

\begin{table}[!htbp]
\centering 
\begin{tabular}{cllll}
\hline
\multicolumn{5}{c}{\bf Fair TME} \\ \hline
$(\mu_1, \mu_2)$ & \multicolumn{2}{c}{TME Errors} & \multicolumn{2}{c}{Fairness Value} \\ \hline 
$(1,1)$  & \multicolumn{2}{c}{[  1.7593,   1.7414,   1.9641,   1.6333]}           & \multicolumn{2}{c}{   0.33075}               \\ 
$(5,1)$  & \multicolumn{2}{c}{[  1.6752,   1.6068,   2.1279,   {\bf1.5500}]}           & \multicolumn{2}{c}{   0.57786}               \\ 
$(1,5)$  & \multicolumn{2}{c}{[  1.8203,   1.8202,   1.8829,   1.7761]}           & \multicolumn{2}{c}{   0.10677}               \\ 
$(10,1)$  & \multicolumn{2}{c}{[  {\bf1.6236},   {\bf1.5384},   2.2060,   1.5653]}           & \multicolumn{2}{c}{   0.66754}               \\ 
$(1,10)$  & \multicolumn{2}{c}{[  1.8362,   1.8367,   {\bf1.8699},   1.8120]}           & \multicolumn{2}{c}{  {\bf 0.05788}}        \\ \hline
\multicolumn{5}{c}{\bf TME} \\ \hline 
& \multicolumn{2}{c}{TME Errors} & \multicolumn{2}{c}{Fairness Value} \\ \hline 
& \multicolumn{2}{c}{[  4.5959,   4.8870,   3.0424,   2.4628]} & \multicolumn{2}{c}{   2.42420}                \\ \hline 
\end{tabular}
\caption{\label{tab:wine}Results of the numerical experiments on the Wine Dataset. The four TME errors in the table correspond to the {\it Red\textunderscore good}, {\it Red\textunderscore bad}, {\it White\textunderscore good} and the {\it White\textunderscore bad} subgroups respectively. Column one in the table gives the parameter settings tested for the Fair TME models.}
\end{table}

\subsubsection{Skills Dataset Experiment} 

The Skills Dataset \cite{SkillData} contains video game tracking data from 3395 players of a real-time strategy game. The four subgroups: {\it Skill\textunderscore Tier\textunderscore1}, {\it Skill\textunderscore Tier\textunderscore2}, {\it Skill\textunderscore Tier\textunderscore3} and {\it Skill\textunderscore Tier 4} were constructed from player rankings. The players were ranked from novice, 1, to professional, 8; {\it Skill\textunderscore Tier\textunderscore1} contains the players with the lowest player rankings of 1 and 2; {\it Skill\textunderscore Tier\textunderscore2} contains players with rankings 3 and 4; {\it Skill\textunderscore Tier\textunderscore3} has the players with rankings 5 and 6, and {\it Skill\textunderscore Tier\textunderscore4} contains the best players with rankings 7 and 8. Due to missing data, we further removed two features.
Fifteen player features remained once categorical features and features with missing entries were removed. Table \ref{tab:skills} displays the results of the numerical experiments. 
\begin{table}[!htbp]
\centering 
\begin{tabular}{cllll}
\hline
\multicolumn{5}{c}{\bf Fair TME} \\ \hline
$(\mu_1, \mu_2)$ & \multicolumn{2}{c}{TME Errors} & \multicolumn{2}{c}{Fairness Value} \\ \hline 
$(1,1)$  & \multicolumn{2}{c}{[  2.0541,   1.5958,   1.5721,   {\bf2.1863}]}           & \multicolumn{2}{c}{   0.61427}               \\ 
$(5,1)$  & \multicolumn{2}{c}{[  1.8855,   {\bf1.0481},   0.9702,   2.4469]}           & \multicolumn{2}{c}{   1.47675}               \\ 
$(1,5)$  & \multicolumn{2}{c}{[  2.1714,   2.0230,   2.0400,   2.1906]}           & \multicolumn{2}{c}{   0.16764}               \\ 
$(10,1)$  & \multicolumn{2}{c}{[  {\bf1.7937},   0.8780,   {\bf0.8212},   2.6447]}           & \multicolumn{2}{c}{   1.82350}               \\ 
$(1,10)$  & \multicolumn{2}{c}{[  2.1967,   2.1144,   2.1286,   2.2046]}           & \multicolumn{2}{c}{  \bf0.09017}        \\ \hline
\multicolumn{5}{c}{\bf TME} \\ \hline 
& \multicolumn{2}{c}{TME Errors} & \multicolumn{2}{c}{Fairness Value} \\ \hline 
& \multicolumn{2}{c}{[  5.2708,   2.0687,   1.5205,   6.0223]} & \multicolumn{2}{c}{   4.50187}                \\ \hline 
\end{tabular}
\caption{\label{tab:skills}Results of the numerical experiments on the Skills Dataset. The four TME errors correspond to the Skill\textunderscore Tier\textunderscore1, Skill\textunderscore Tier\textunderscore2, Skill\textunderscore Tier\textunderscore3 and Skill\textunderscore Tier\textunderscore4 subgroups respectively. Column one gives the parameter setting for the Fair TME models.}
\end{table}

The results from Table \ref{tab:skills} demonstrate the Fair TME model again outperformed the standard TME model. For the strictest parameter setting promoting fairness, the Fair TME model obtained a fairness value approximately 50 times smaller than the fairness value achieved by the TME model. Additionally, the best TME errors produced by the Fair TME model for the different parameter settings were better than the TME errors produced by the standard model. For the parameter setting $(\mu_1, \mu_2) = (10,1)$ all of the TME errors and the fairness value produced by the Fair TME model were smaller than the corresponding values produced by the standard TME model.

\subsubsection{Credit Dataset Experiment} 

In the final experiment we used the default credit card dataset \cite{yeh2009comparisons} which contains credit information from Taiwan for 30,000 individuals. From the categorical elements we formed eight different subgroups based on gender, marital status and education. An individual was considered to have a high education (abbreviated HgEd) if they obtained at least a college degree otherwise they were labeled as having a low education (abbreviated LwEd). Thus, the subgroup {\it Male\textunderscore Married\textunderscore HgEd} is comprised of married men with college degrees, and {\it Female\textunderscore Single\textunderscore LwEd} contains single women without college degrees. Upon removing the categorical information from the dataset 19 features remained which tracked past payment history, amount on the billing statement over several months and the total amount of given credit. The results of the experiment are provided in Table \ref{tab:credit}.

\begin{table}[ht]
\centering
\tabcolsep=0.11cm
\begin{tabular}{cllll}
\hline
\multicolumn{5}{c}{\bf Fair TME} \\ \hline
$(\mu_1, \mu_2)$ & \multicolumn{2}{c}{TME Errors} & \multicolumn{2}{c}{Fairness Value} \\ \hline 
$(1,1)$  & \multicolumn{2}{c}{[  0.8602,   {\bf1.3988},   1.1714,   1.3107,   1.1070,   1.1968,   1.1247,   {\bf 1.4034}]}           & \multicolumn{2}{c}{   0.54313}               \\ 
$(5,1)$  & \multicolumn{2}{c}{[  0.5514,   1.4438,   0.8200,   1.2548,   0.8371,   1.0563,   0.9476,   1.5759]}           & \multicolumn{2}{c}{   1.02451}               \\ 
$(1,5)$  & \multicolumn{2}{c}{[  1.2453,   1.4579,   1.4234,   1.4285,   1.3709,   1.3928,   1.3521,   1.4477]}           & \multicolumn{2}{c}{   0.21268}               \\ 
$(10,1)$ & \multicolumn{2}{c}{[  {\bf0.4796},   1.4752,   0.7089,   {\bf1.2506},   0.7676,   {\bf 1.0286},   0.9072,   1.6790]}           & \multicolumn{2}{c}{   1.19937}              \\ 
$(1,10)$ & \multicolumn{2}{c}{[  1.3780,   1.5022,   1.4896,   1.4854,   1.4550,   1.4665,   1.4399,   1.4957]}           & \multicolumn{2}{c}{\bf 0.12419}       \\ \hline
\multicolumn{5}{c}{\bf TME} \\ \hline 
& \multicolumn{2}{c}{TME Errors} & \multicolumn{2}{c}{Fairness Value} \\ \hline 
& \multicolumn{2}{c}{[  0.7839,   2.4049,   {\bf0.5246},   1.7299,   {\bf0.3924},   1.4689,   {\bf0.5209},   2.1529]} & \multicolumn{2}{c}{   2.01253}                \\ \hline 
\end{tabular}
\caption{\label{tab:credit}Results of the numerical experiments on the Credit Dataset. The eight TME errors in the table correspond to the Male\textunderscore Single\textunderscore HgEd, Male\textunderscore Single\textunderscore LwEd, Male\textunderscore Married\textunderscore HgEd, Male\textunderscore Married\textunderscore LwEd, Female\textunderscore Single\textunderscore HgEd, Female\textunderscore Single\textunderscore LwEd, Female\textunderscore Married\textunderscore HgEd, Female\textunderscore Married\textunderscore LwEd subgroups respectively.}
\end{table}

The results in Table \ref{tab:credit} again affirm our model's advantage for fairness. The best fairness value produced by the Fair TME model was 16 times smaller than the fairness value generated by the standard TME model. Also, as in the Wine and Skills datasets, the TME errors produced by the Fair TME model were often smaller than those produced by the TME model. For parameter setting $(\mu_1, \mu_2) = (10,1)$ the TME errors produced by the Fair TME model were lower than those of the standard TME model for all but two of the eight subgroups. 

Another observation to note is the standard TME model was again unfair to underrepresented subgroups - the subgroups with the fewest number of samples. The largest subgroup in this example was the {\it Female\textunderscore Single\textunderscore HgEd} subgroup with 8260 samples and its TME error was 0.3924 while the smallest subgroup was the {\it Male\textunderscore Single\textunderscore LwEd} subgroup with 974 samples and its corresponding TME error was 2.4049. This phenomenon was also present in the Skills Dataset between the {\it Skill\textunderscore Tier\textunderscore1} and {\it Skill\textunderscore Tier\textunderscore4} subgroups. Therefore, one apparent advantage of the Fair TME model over the standard TME model is its lack of favoritism to subgroups with a greater relative number of points in the dataset.

\section{Conclusions}\label{sec:conclusions}

We utilized matrix functions to rewrite a constrained matrix optimization model as an unconstrained model over the set of symmetric matrices. We then proposed to solve the non-convex unconstrained model with cubic-regularized Newton and proved a chain rule identity which enabled us to compute the necessary higher-order derivatives to implement the method. Finally, we tested the approach on a fair version of the Tyler's M-estimator model, the Fair TME model, which sought to locate a fair and robust covariance matrix for multiple subgroups of a single dataset. Through experiments on real and simulated data we demonstrated the Fair TME model produced solutions with strong TME errors and fairness values which were superior to those obtained by the standard TME formulation.  

Moving forward there are several ways to extend the outlined method. First, the key step in implementing Algorithm \ref{alg:CRN} is solving the matrix subproblem in Step 2. Recently efficient methods for solving the cubic-regularized subproblem in the vector setting have drawn research attention \cite{C19,jiang2022cubic,jiang2021accelerated}; however, not as much work has been done on the matrix problem. Investigating how to extend these solution methods to the matrix setting without vectorizing the problem could yield efficient subroutines for large matrix problems.

Further, at this time we restricted our discussion to constrained matrix optimization problems over symmetric matrices and we only investigated one type of matrix function; however, another line of research would be to see what types of constrained problems could be handled with generalized matrix functions for non-square matrices \cite{hawkins1973generalized} or even tensor functions \cite{li2021continuity}. This could enable higher-order methods to be generated solving non-square matrix and tensor optimization models which are becoming more prevalent in applications

\section*{Funding}
This material is based upon work supported by the National Science Foundation Graduate Research Fellowship Program under Grant No. 1839286 and the National Science Foundation award DMS-2124913. Any opinions, findings, and conclusions or recommendations expressed in this material are those of the author(s) and do not necessarily reflect 
the views of the National Science Foundation 



\section{Appendix A: Proof of Lemma \ref{lem:hess_lip}}\label{sec:Appendix_B}

\begin{lemma}\label{lem:hess_lip_a}
If the Hessian of $F: \mathcal{V} \rightarrow \R$ is Lipschitz continuous on $\mathcal{V}$ with parameter $L>0$, then for all $\bX,\bY \in \mathcal{V}$, 
\begin{equation}\label{eqn:lip_ineq_1_a}
\| \bnabla F(\bY) - \bnabla F(\bX) - \bnabla^2 F(\bX)[\bY-\bX] \| _F \leq \frac{L}{2} \| \bX - \bY \|_F^2,
\end{equation}
\begin{equation}\label{eqn:lip_ineq_2_a}
| F(\bY) - F(\bX)- \langle \nabla F(\bX), \bY-\bX \rangle - \frac{1}{2}\langle \bnabla^2 F(\bX)[\bY-\bX], \bY-\bX \rangle| \leq \frac{L}{6} \| \bX - \bY \|_F^3.
\end{equation}
\end{lemma}

\begin{proof}{ Define $g(t) := \bnabla F(\bX + t(\bY-\bX))$. By the Fundamental Theorem of Calculus, 
\[
\bnabla F(\bY) - \bnabla F(\bX) = \int_{0}^{1} \; \bnabla^2 F(\bX+t(\bY-\bX))[\bY-\bX]\; dt.
\]
Therefore, 
\begin{align}
\| \bnabla F(\bY) &- \bnabla F(\bX) - \bnabla^2 F(\bX)[\bY-\bX]\|_F  \nonumber \\
&\leq \| \int_{0}^{1}  \bnabla^2 F(\bX+t(\bY-\bX))[\bY-\bX] - \bnabla^2 F(\bX)[\bY-\bX] \; dt  \|_F \nonumber \\
&\leq \int_{0}^{1} L t \|\bX-\bY \|_F^2 \; dt \nonumber \\
&= \frac{L}{2} \|\bX-\bY \|_F^2. \nonumber 
\end{align}
By Taylor's theorem and \eqref{eqn:lip_ineq_1_a},
\begin{align}
|F(\bY) - &F(\bX) - \langle \nabla F(\bX), \bY-\bX \rangle - \frac{1}{2}\langle \bnabla^2 F(\bX)[\bY-\bX], \bY-\bX \rangle| \nonumber \\
&\leq  \int_{0}^{1} | \langle \bnabla F(\bX + t(\bY-\bX)) - \nabla F(\bX) - t \bnabla^2 F(\bX)[\bY-\bX], \bY-\bX \rangle| \; dt \nonumber \\
&\leq \int_{0}^{1} \| \bnabla F(\bX + t(\bY-\bX)) - \bnabla F(\bX) - t \bnabla^2 F(\bX)[\bY-\bX]\|_F \|\bY-\bX\|_F \; dt \nonumber \\
&\leq  \int_{0}^{1} \frac{L}{2} \|t(\bY-\bX)\|_F^2 \|\bY-\bX\|_F \; dt \nonumber \\
&= \frac{L}{6}\|\bY-\bX\|_F^3. \nonumber 
\end{align}
} 
\vspace{-0.25in}
\end{proof} 


\section{Appendix B: Proof of Theorem \ref{thm:CR}}\label{sec:Appendix_C} We prove Theorem \ref{thm:CR} restated below. 
\begin{theorem}\label{thm:CR_a}
Let $\mathcal{V} \subseteq \R^{n \times n}$ be a vector space, $F: \mathcal{V} \rightarrow \R$ and $\bG: \mathcal{V} \rightarrow \mathcal{V}$ be any matrix function defined by an analytic scalar function. Let $\bX, \bH \in \mathcal{V}$ and assume $F$ is differentiable at $\bG(\bX)$ and $\bG$ is differentiable at $\bX$. Then, assuming $F$ and $\bG$ are two-times continuously differentiable, the first and second-order directional derivatives of $F \circ \bG$ at $\bX$ in the direction of $\bH$ are:
\begin{equation}\label{eqn:CR_1st_a}
D(F\circ \bG)(\bX)[\bH] = DF(\bG(\bX))\bigg[ D\bG(\bX)[\bH] \bigg],
\end{equation} 
\begin{multline}\label{eqn:CR_2nd_a}
D^2(F\circ \bG)(\bX)[\bH,\bH] = \\ D^2F(\bG(\bX))\bigg[ D\bG(\bX)[\bH], D\bG(\bX)[\bH]\bigg] + DF(\bG(\bX))\bigg[D^2 \bG(\bX)[\bH,\bH]\bigg].
\end{multline} 
\end{theorem}
\noindent The first order chain rule identity has been proven in multiple references for general vector spaces (e.g. \cite{Lee}). Thus, we only verify the second identity, and this will be accomplished by proving \eqref{eqn:CR_2nd_a} holds for arbitrary matrix polynomials, 
\begin{equation}\label{eqn:mat_poly}
\bG(\bX) = \sum_{i=0}^{n} a_i \bX^i,
\end{equation}
where $a_i \in \R$ and $\bG: \R^{n\times n} \rightarrow \R^{n\times n}$. Proving the identity holds for arbitrary matrix polynomials will be sufficent because $\bG$ is generated by an analytic scalar function. Note, if $\bG$ is generated by the analytic scalar function $G$, then without loss of generality we can assume the scalar function $G$ has an infinite Taylor series expansion about zero, 
$$
G(x) = \sum_{i=0}^{\infty} a_0 x^i.
$$
Then, by Definition \ref{def:mat_func}, for $\bX = \bQ \Diag(\lambda_1, \hdots, \lambda_n)\bQ^\top$, 
\[
\bG(\bX) = \bQ \bG(\bLambda) \bQ^\top = \bQ \left( \sum_{i=0}^{\infty} a_i \bLambda^i\right) \bQ^\top  = \sum_{i=0}^{\infty} a_i \bX^i.
\]
Therefore, if we prove \eqref{eqn:CR_2nd_a} holds for any matrix polynomial \eqref{eqn:mat_poly} of arbitrary degree, then it will hold true for any matrix function generated by an analytic scalar function. We now prove a few intermediate results which will allow us to prove Theorem \ref{thm:CR_a}.
\begin{lemma}\label{lem:1}
If \;$\bG(\bX) = \bX^n$, then for $\ell_1, \ell_2, \ell_3 \in \{0,1,2,\hdots\}$ 
\begin{equation}\label{eqn:lem3_eq1}
\bD\bG(\bX)[\bH] = \sum_{\ell_1+\ell_2= n -1} \bX^{\ell_1} \bH \bX^{\ell_2},
\end{equation}
\begin{equation}\label{eqn:lem3_eq2}
\frac{1}{2}\bD^2 \bG(\bX)[\bH,\bH] = \sum_{\ell_1+\ell_2 = n-2} \bX^{\ell_1} \bH^2 \bX^{\ell_2} + \sum_{\ell_1 + \ell_2 + \ell_3 = n-3} \bX^{\ell_1} \bH \bX^{\ell_2+1} \bH \bX^{\ell_3}. 
\end{equation}
\end{lemma}
\begin{proof}{We prove these identities by induction. Note, 
\begin{align}
(\bX + \bH )\; &= \bX + \bH, \nonumber \\
(\bX + \bH )^2 &= \bX^2 + \bX\bH + \bH\bX + \bH^2, \nonumber \\
(\bX + \bH )^3 &= \bX^3 + \bX^2 \bH + \bX\bH\bX + \bX\bH^2 + \bH\bX^2 + \bH\bX\bH + \bH^2 \bX + \bH^3, \nonumber 
\end{align}
from which we deduce: 
\begin{align}
\bD(\bX^1)(\bX)[\bH] &= \bH, \;\; &&\frac{1}{2}\bD^2(\bX^1)(\bX)[\bH,\bH] = 0, \nonumber \\
\bD(\bX^2)(\bX)[\bH] &= \bX\bH + \bH\bX, \;\; &&\frac{1}{2}\bD^2(\bX^2)(\bX)[\bH,\bH] = \bH^2, \nonumber \\
\bD(\bX^3)(\bX)[\bH] &= \bX^2 \bH + \bX\bH\bX  + \bH\bX^2, \;\; &&\frac{1}{2}\bD^2(\bX^3)(\bX)[\bH,\bH] = \bH^2 \bX + \bH\bX\bH + \bX \bH^2.\nonumber 
\end{align}
Thus, we have our base cases. Assume \eqref{eqn:lem3_eq1} and \eqref{eqn:lem3_eq2} hold for all $n$ up to some $k \geq 3$. Then, 
\begin{align}
(\bX + \bH)^{k+1}  &= (\bX+\bH)\left(\bX^k + \bD(\bX^k)(\bX)[\bH] + \mathcal{O}(\bH^2) \right) \nonumber \\
&= (\bX+\bH) \left( \bX^k + \sum_{\ell_1+\ell_2= k -1} \bX^{\ell_1} \bH \bX^{\ell_2}+ \mathcal{O}(\bH^2) \right) \nonumber \\
				  &= \bX^{k+1} +\sum_{\ell_1+\ell_2= k -1} \bX^{\ell_1+1} \bH \bX^{\ell_2}+ \bH\bX^k +  \mathcal{O}(\bH^2) \nonumber \\
				  &= \bX^{k+1} +  \sum_{\ell_1+\ell_2= k} \bX^{\ell_1} \bH \bX^{\ell_2} + \mathcal{O}(\bH^2), \nonumber 
\end{align}
where the second equality follows from the inductive hypothesis. The final equality implies\\ $\bD(\bX^{k+1})(\bX)[\bH] =  \sum_{\ell_1+\ell_2= k} \bX^{\ell_1} \bH \bX^{\ell_2} $ proving \eqref{eqn:lem3_eq1} by induction. Similarly, 
\begin{align}
&(\bX+\bH)^{k+1} =(\bX+\bH)\left(\bX^k + \bD(\bX^k)(\bX)[\bH] + \frac{1}{2}\bD^2(\bX^k)(\bX)[\bH,\bH] + \mathcal{O}(\bH^3) \right) \nonumber \\
&=(\bX+\bH)\hspace{-0.05in}\left( \bX^k + \hspace{-0.20in}  \sum_{\ell_1+\ell_2= k -1} \hspace{-0.15in} \bX^{\ell_1} \bH \bX^{\ell_2} + \hspace{-0.20in} \sum_{\ell_1+\ell_2 = k-2} \hspace{-0.15in} \bX^{\ell_1} \bH^2 \bX^{\ell_2} + \hspace{-0.25in} \sum_{\ell_1 + \ell_2 + \ell_3 = k-3} \hspace{-0.25in}  \bX^{\ell_1} \bH \bX^{\ell_2+1} \bH \bX^{\ell_3} + \mathcal{O}(\bH^3) \right) \nonumber \\ \nonumber \\
&=\bX^{k+1} + \sum_{\ell_1+\ell_2= k } \bX^{\ell_1} \bH \bX^{\ell_2} + {\color{blue}\sum_{\ell_1+\ell_2 = k-2} \bX^{\ell_1+1} \bH^2 \bX^{\ell_2}}
\nonumber \\
&\hspace{1.15in}+{\color{red}\sum_{\ell_1 + \ell_2 + \ell_3 = k-3}  \bX^{\ell_1 +1} \bH \bX^{\ell_2+1} \bH \bX^{\ell_3}}
+ {\color{blue}\sum_{\ell_1+\ell_2= k -1} \bH \bX^{\ell_1} \bH \bX^{\ell_2}} +  \mathcal{O}(\bH^3). \nonumber
\end{align}
Note, 
\[
{\color{blue}\sum_{\ell_1+\ell_2 = k-2} \hspace{-0.2in}\bX^{\ell_1+1} \bH^2 \bX^{\ell_2} + \sum_{\ell_1+\ell_2= k -1} \hspace{-0.2in}\bH \bX^{\ell_1} \bH \bX^{\ell_2}} = \sum_{\ell_1+\ell_2=k-1} \hspace{-0.2in}\bX^{\ell_1} \bH^2 \bX^{\ell_2} + {\color{magenta}\sum_{\ell_1+\ell_2= k - 2} \hspace{-0.2in}\bH \bX^{\ell_1+1} \bH \bX^{\ell_2}},
\]
and 
\[
{\color{red} \sum_{\ell_1 + \ell_2 + \ell_3 = k-3} \hspace{-0.2in} \bX^{\ell_1 +1} \bH \bX^{\ell_2+1} \bH \bX^{\ell_3}} +  {\color{magenta}\sum_{\ell_1 + \ell_2= k - 2} \hspace{-0.2in}\bH \bX^{\ell_1+1} \bH \bX^{\ell_2}} = \sum_{\ell_1 + \ell_2 + \ell_3 = k-2} \hspace{-0.2in} \bX^{\ell_1} \bH \bX^{\ell_2+1} \bH \bX^{\ell_3}.
\]
Therefore, 
\begin{multline}
(\bX + \bH)^{k+1} =  \bX^{k+1} + \sum_{\ell_1+\ell_2= k } \bX^{\ell_1} \bH \bX^{\ell_2} +  \sum_{\ell_1+\ell_2=k-1} \bX^{\ell_1} \bH^2 \bX^{\ell_2}  \nonumber \\ \hspace{-0.6in} + \sum_{\ell_1 + \ell_2 + \ell_3 = k-2}  \bX^{\ell_1} \bH \bX^{\ell_2+1} \bH \bX^{\ell_3} + \mathcal{O}(\bH^3), 
\end{multline}
and \eqref{eqn:lem3_eq2} follows by induction.}
\end{proof}

Lemma \ref{lem:1} by the linearity of differential operators leads in a straightforward fashion to the following corollary which provides the first and second-order directional derivatives of all finite matrix polynomials. 
\begin{corollary}\label{thm:cor1}
If \;$\bG(\bX) = \sum_{i=0}^{n} a_i \bX^n$, then, 
\[
\bD\bG(\bX)[\bH] = \sum_{i=1}^{n}\;\; \sum_{\ell_1+\ell_2= i -1} a_i \bX^{\ell_1} \bH \bX^{\ell_2},
\]
\[
\frac{1}{2}\bD^2 \bG(\bX)[\bH,\bH] =  \sum_{i=2}^{n} a_i \left( \sum_{\ell_1+\ell_2= i-2} \bX^{\ell_1} \bH^2 \bX^{\ell_2} + \sum_{\ell_1 + \ell_2 + \ell_3 = i-3} \bX^{\ell_1} \bH \bX^{\ell_2+1} \bH \bX^{\ell_3}\right). 
\]
\end{corollary}

In the next two lemmas, we compute the gradient and Hessian of the composition $F \circ \bG$ which we will utilize in our proof of Theorem \ref{thm:CR_a}. 
\begin{lemma}\label{lem:grad}
If $F$ is as defined in Theorem \ref{thm:CR_a} and $\bG$ is as defined in \eqref{eqn:mat_poly}, then, 
\[
\bnabla( F\circ \bG)(\bX) = \sum_{i=1}^{n} \;\; \sum_{\ell_1 + \ell_2  =i-1} a_i (\bX^{\ell_1})^\top \bnabla F(\bG(\bX)) (\bX^{\ell_2})^\top. 
\]
\end{lemma}
\begin{proof}\; By the first chain rule identity in Theorem \ref{thm:CR} and Corollary \ref{thm:cor1}, 
\begin{align}
D(F\circ \bG)(\bX)[\bH] &= D(F(\bG(\bX))\bigg[\bD\bG(\bX)[\bH]\bigg] \nonumber \\
				      &= \langle \bnabla F (\bG(\bX)),\; \bD\bG(\bX)[\bH] \rangle \nonumber \\
				      &= \sum_{i=1}^{n}\;\; \sum_{\ell_1+\ell_2= i -1} \bigg\langle \bnabla F (\bG(\bX)),  a_i \bX^{\ell_1} \bH \bX^{\ell_2} \bigg \rangle \nonumber \\
				      &=\sum_{i=1}^{n}\;\; \sum_{\ell_1+\ell_2= i -1} \bigg\langle a_i  (\bX^{\ell_1})^\top \bnabla F (\bG(\bX)) (\bX^{\ell_2})^\top, \bH \bigg \rangle, \nonumber 
\end{align}
and the last equality follows since $\langle \bm{A}, \bX^{\ell_1} \bH \bX^{\ell_2} \rangle = \langle (\bX^{\ell_1})^\top \bm{A} (\bX^{\ell_2})^\top, \bH \rangle$ for all $\bm{A}, \bX, \bH \in \R^{n\times n}$. \hfill 
\end{proof}

\begin{lemma}\label{lem:hess}
If $F$ is as defined in Theorem \ref{thm:CR_a} and $\bG$ is as defined in \eqref{eqn:mat_poly}, then, 
\begin{align}\label{eqn:hess_fcircg_1}
\bnabla^2 (F \circ \bG)(\bX)[\bH]  =&\sum_{i=2}^{n} \;\; \sum_{\ell_1 + \ell_2 + \ell_3 =i-2} a_i (\bX^{\ell_1})^\top \bH^\top (\bX^{\ell_2})^\top  \bnabla F(\bG(\bX)) (\bX^{\ell_3})^\top \nonumber \\
&+\sum_{i=2}^{n} \;\; \sum_{\ell_1 + \ell_2 + \ell_3 =i-2} a_i (\bX^{\ell_1})^\top \bnabla F(\bG(\bX)) (\bX^{\ell_2})^\top  \bH^\top (\bX^{\ell_3})^\top \nonumber \\
&+ \sum_{i=1}^{n} \;\; \sum_{\ell_1 + \ell_2  =i-1} a_i  (\bX^{\ell_1})^\top \bnabla^2 F(\bG(\bX))\bigg[\bD\bG(\bX)[\bH] \bigg] (\bX ^{\ell_2})^\top. 
\end{align}
\end{lemma}
\begin{proof}\;
By definition, 
\begin{align}
\bnabla^2 (F \circ \bG)(\bX)[\bH] =&\frac{d}{dt} \left( \bnabla ( F\circ \bG)(\bX + t\bH) \right)\bigg|_{t=0} \nonumber \\
=&\sum_{i=1}^{n} \;\; \sum_{\ell_1 + \ell_2  =i-1} a_i  \frac{d}{dt} \left[ \left\{(\bX + t\bH)^{\ell_1}\right\}^\top \right]\bigg|_{t=0} \bnabla F(\bG(\bX)) (\bX^{\ell_2})^\top \nonumber \\
&+ \sum_{i=1}^{n} \;\; \sum_{\ell_1 + \ell_2  =i-1} a_i  (\bX^{\ell_1})^\top \bnabla F(\bG(\bX)) \frac{d}{dt} \left[ \left\{(\bX + t\bH)^{\ell_2}\right\}^\top \right]\bigg|_{t=0} \nonumber \\
&+ \sum_{i=1}^{n} \;\; \sum_{\ell_1 + \ell_2  =i-1} a_i  (\bX^{\ell_1})^\top \bnabla^2 F(\bG(\bX))\bigg[\bD\bG(\bX)[\bH] \bigg] (\bX ^{\ell_2})^\top. \nonumber
\end{align}
Note, in deriving this equation we utilized Lemma \ref{lem:grad}, the product rule and the fact, 
\[
\frac{d}{dt} \left[ \bnabla F( \bG(\bX + t\bH)) \right]\bigg|_{t=0} =  \bnabla^2 F(\bG(\bX))\bigg[\bD\bG(\bX)[\bH] \bigg]. 
\]

\noindent Investigating the first term, using Lemma \ref{lem:1} to compute the scalar derivative, we see, 
\begin{align}
\sum_{i=1}^{n} \;\; \sum_{\ell_1 + \ell_2  =i-1} a_i  &\frac{d}{dt} \left[ \left\{(\bX + t\bH)^{\ell_1}\right\}^\top \right]\bigg|_{t=0} \bnabla F(\bG(\bX)) (\bX^{\ell_2})^\top \nonumber \\
&= \sum_{i=2}^{n} \;\; \sum_{\ell_1 + \ell_2  =i-1} a_i  \frac{d}{dt} \left[ \left\{(\bX + t\bH)^{\ell_1}\right\}^\top \right]\bigg|_{t=0} \bnabla F(\bG(\bX)) (\bX^{\ell_2})^\top \nonumber \\
&= \sum_{i=2}^{n} \;\; \sum_{\ell_1 + \ell_2  =i-2} a_i  \frac{d}{dt} \left[ \left\{(\bX + t\bH)^{\ell_1 + 1}\right\}^\top \right]\bigg|_{t=0} \bnabla F(\bG(\bX)) (\bX^{\ell_2})^\top \nonumber \\
&= \sum_{i=2}^{n} \;\; \sum_{\ell_1 + \ell_2  =i-2} a_i  \left( \sum_{k_1 + k_2 = \ell_1} (\bX^{k_1})^\top \bH^\top (\bX^{k_2})^\top  \right)\bnabla F(\bG(\bX)) (\bX^{\ell_2})^\top \nonumber \\
&= \sum_{i=2}^{n} \;\; \sum_{\ell_1 + \ell_2 + \ell_3 =i-2} a_i (\bX^{\ell_1})^\top \bH^\top (\bX^{\ell_2})^\top  \bnabla F(\bG(\bX)) (\bX^{\ell_3})^\top. \nonumber
\end{align}
Rewriting the second term in a similar fashion we obtain \eqref{eqn:hess_fcircg_1}. \hfill 
\end{proof}

We now prove the second-order chain rule identity in Theorem \ref{thm:CR_a} with these results. 

\noindent{\it Proof of Theorem \ref{thm:CR_a}.}\; To prove \eqref{eqn:CR_2nd_a}, we only need to demonstrate, 
\begin{multline}\label{eqn:main_eq}
\langle \bnabla^2 (F \circ \bG)(\bX)[\bH], \bH \rangle \\ =  D^2F(\bG(\bX))\bigg[ \bD\bG(\bX)[\bH], \bD\bG(\bX)[\bH]\bigg] + DF(\bG(\bX))\bigg[\bD^2 \bG(\bX)[\bH,\bH]\bigg].
\end{multline}

\noindent The left-hand side of \eqref{eqn:main_eq} is known from Lemma \ref{lem:hess}; therefore, we compute the right-hand side to verify the equality. Looking at the first term, 

\begin{align}\label{eqn:rhs_term1}
 D^2F(\bG(\bX))\bigg[ &\bD\bG(\bX)[\bH], \bD\bG(\bX)[\bH]\bigg] \nonumber \\
 &\overset{}{=} \bigg \langle \bnabla^2 F(\bG(\bX))\bigg[\bD\bG(\bX)[\bH]\bigg], \bD\bG(\bX)[\bH] \bigg \rangle \nonumber \\
 &\overset{}{=}  \bigg \langle \bnabla^2 F(\bG(\bX))\bigg[\bD\bG(\bX)[\bH]\bigg], \sum_{i=1}^{n}\;\; \sum_{\ell_1+\ell_2= i -1} a_i \bX^{\ell_1} \bH \bX^{\ell_2} \bigg \rangle \nonumber \\
 &=  \bigg \langle  \sum_{i=1}^{n}\;\; \sum_{\ell_1+\ell_2= i -1} a_i (\bX^{\ell_1})^\top \bnabla^2 F(\bG(\bX))\bigg[\bD\bG(\bX)[\bH]\bigg] (\bX^{\ell_2})^\top,  \bH  \bigg \rangle.  
\end{align}

\noindent Looking at the second term, 
\begin{align}\label{eqn:2ndRHSinterm}
&DF(\bG(\bX))\bigg[\bD^2 \bG(\bX)[\bH,\bH]\bigg] \nonumber \\
&= \bigg \langle \bnabla F(\bG(\bX)),\; D^2 \bG(\bX)[\bH,\bH] \bigg \rangle \nonumber \\
&= \bigg \langle \bnabla F(\bG(\bX)), \;2  \sum_{i=2}^{n}\;\; \sum_{\ell_1+\ell_2= i-2}\hspace{-0.15in} a_i  \bX^{\ell_1} \bH^2 \bX^{\ell_2} + 2 \sum_{i=2}^{n}\;\;\; \sum_{\ell_1 + \ell_2 + \ell_3 = i-3} \hspace{-0.15in} a_i \bX^{\ell_1} \bH \bX^{\ell_2+1} \bH \bX^{\ell_3} \bigg \rangle. 
\end{align}

\noindent From the properties of the trace we note for all $\ell_1$ and $\ell_2$, 
\begin{align}
\bigg \langle \bnabla F(\bG(\bX)), \bX^{\ell_1} \bH^2 \bX^{\ell_2} \bigg \rangle &= \bigg \langle \bH^\top (\bX^{\ell_1})^\top \bnabla F(\bG(\bX)) (\bX^{\ell_2})^\top, \bH \bigg \rangle  \nonumber \\
&=\bigg \langle (\bX^{\ell_1})^\top \bnabla F(\bG(\bX)) (\bX^{\ell_2})^\top \bH^\top, \bH \bigg \rangle, \nonumber 
\end{align}
and 
\begin{align}
\bigg \langle \bnabla F(\bG(\bX)), \bX^{\ell_1} \bH \bX^{\ell_2} \bH \bX^{\ell_3} \bigg \rangle &= \bigg \langle (\bX^{\ell_2})^\top \bH^\top (\bX^{\ell_1})^\top \bnabla F(\bG(\bX)) (\bX^{\ell_3})^\top, \bH \bigg \rangle  \nonumber \\
&= \bigg \langle  (\bX^{\ell_1})^\top \bnabla F(\bG(\bX)) (\bX^{\ell_3})^\top\bH^\top (\bX^{\ell_2})^\top,  \bH \bigg \rangle. \nonumber 
\end{align}
Using these  alternative manners for expressing these inner products we rewrite \eqref{eqn:2ndRHSinterm} as, 
\begin{align}\label{eqn:term2_RHS_final}
DF(\bG(\bX))\bigg[\bD^2 \bG(\bX)[\bH,\bH]\bigg]=&\bigg \langle \sum_{i=2}^{n}\;\; \sum_{\ell_1+\ell_2= i-2} a_i \bH^\top (\bX^{\ell_1})^\top \bnabla F(\bG(\bX)) (\bX^{\ell_2})^\top, \bH \bigg \rangle \nonumber \\
&+ \bigg \langle \sum_{i=2}^{n}\;\; \sum_{\ell_1+\ell_2= i-2} a_i (\bX^{\ell_1})^\top \bnabla F(\bG(\bX)) (\bX^{\ell_2})^\top \bH^\top, \bH \bigg \rangle \nonumber \\
&+ \bigg \langle \sum_{i=2}^{n}\;\; \sum_{\ell_1 + \ell_2 + \ell_3 = i-3} \hspace{-0.25in} a_i (\bX^{\ell_2 + 1})^\top \bH^\top (\bX^{\ell_1})^\top \bnabla F(\bG(\bX)) (\bX^{\ell_3})^\top, \bH \bigg \rangle \nonumber \\
&+ \bigg \langle \sum_{i=2}^{n}\;\; \sum_{\ell_1 + \ell_2 + \ell_3 = i-3} \hspace{-0.25in} a_i (\bX^{\ell_1})^\top  \bnabla F(\bG(\bX)) (\bX^{\ell_3})^\top \bH^\top (\bX^{\ell_2 + 1})^\top, \bH \bigg \rangle \nonumber \\
=&\bigg \langle \sum_{i=2}^{n}\;\; \sum_{\ell_1+\ell_2 + \ell_3= i-2} \hspace{-0.15in} a_i (\bX^{\ell_1})^\top \bH^\top (\bX^{\ell_2})^\top \bnabla F(\bG(\bX)) (\bX^{\ell_3})^\top, \bH \bigg \rangle \nonumber \\
&+\bigg \langle \sum_{i=2}^{n}\;\; \sum_{\ell_1+\ell_2 + \ell_3= i-2} \hspace{-0.15in} a_i  (\bX^{\ell_1})^\top \bnabla F(\bG(\bX)) (\bX^{\ell_2})^\top \bH^\top (\bX^{\ell_3})^\top, \bH \bigg \rangle.
\end{align}
Thus, from \eqref{eqn:hess_fcircg_1}, \eqref{eqn:rhs_term1} and  \eqref{eqn:term2_RHS_final} we obtain \eqref{eqn:main_eq}. $\hfill \square$


\section{Appendix C: Proofs of Theorems \ref{thm:global_convergence_CRN} and \ref{thm:superlinear_CRN}}\label{sec:Appendix_D} 
We provide a new analysis for the convergence of Algorithm \ref{alg:CRN}. The main subproblem required at each iteration takes the form, 
\begin{equation}\label{eqn:sub_probH}
\min_{\bH \in \mathcal{V}}\; \langle \bnabla F(\bX), \bH \rangle + \frac{1}{2} \langle \bnabla^2 F(\bX)[\bH], \bH \rangle + \frac{M}{6} \|\bH\|_F^3 ,
\end{equation}
where $M>0$. Define $\bH_X$ to be a global minimizer of \eqref{eqn:sub_probH}. We are going to express this subproblem in an equivalent vectorized form for our analysis. We first state the following result which relates the vectorization of $\bH \in \mathcal{V}$ to vectors in $\R^d$
 
\begin{theorem}\label{thm:Amat}
Let $\mathcal{V} \subseteq \mathbb{R}^{n\times n}$ be a vector space of dimension $d \leq n^2$, then there exists a matrix $\bB \in \R^{n^2 \times d}$ such that $\bB^\top \bB = \bm{I}_{d \times d}$ and, 
$
\left\{ \bB \bz \; | \; \bz \in \R^{d} \right\} = \left\{ \vect(\bH) \; | \; \bH \in \mathcal{V} \right\}.
$
\end{theorem}
\begin{proof}\; Let $V := \left\{ \vect(\bH) \; | \; \bH \in \mathcal{V} \right\} \subseteq \R^{n^2}$. Then it is easy to check $V$ is a vector subspace of $\R^{n^2}$ with dimension $d$. Let $\{\bb^1, \hdots, \bb^d\} \subseteq \R^{n^2}$ be a basis for $V$ and without loss of generality assume it is an orthonormal basis. Define $\bB$ to be the matrix whose columns are the basis vectors of $V$, i.e. 
\[ 
\bB := \left[ \bb^1 \; | \; \hdots \; | \; \bb^d \right].
\]
By the orthogonality of the basis, we see $\bB^\top \bB = \bm{I}_{d \times d}$ and clearly $\text{Range}(\bB) = V$ by the definition of a basis since, 
\[ 
\bB\bz = \sum_{i=1}^{d} z_i \bb^i.
\] \hfill 
\end{proof}
\vspace{-0.2in}
\noindent We present a simple example for such a matrix $\bB$ in the case of symmetric matrices. 

\begin{example}
If $\cV = \mathcal{S}^{3 \times 3}$, then matrix $\bB$ from Theorem \ref{thm:Amat} could be formed by vectorizing and normalizing the six matrices which form the standard basis for $\mathcal{S}^{3 \times 3}$, i.e. 
\[
\bB = \begin{pmatrix}
	1     &0                  &0     &0     &0     &0 \\
     0       &\sqrt{2}/2     &0     &0     &0     &0 \\
      0     &0                   &\sqrt{2}/2      &0     &0     &0 \\
       0    &\sqrt{2}/2      &0     &0     &0     &0 \\
        0   &0                  &0     &1     &0     &0 \\
         0 &0     &0     &0     &\sqrt{2}/2     &0 \\ 
         0 &0     &\sqrt{2}/2      &0     &0     &0 \\
         0 &0     &0     &0     &\sqrt{2}/2     &0 \\
         0 &0     &0     &0     &0     &1 
	  \end{pmatrix},
\]
where the second and sixth columns, for example, are the normalized versions of the vectors, 
\begin{align}
\vect\left( \begin{pmatrix} 0 & 1 & 0 \\ 1 & 0 & 0 \\ 0 & 0 & 0 \end{pmatrix} \right) = \left( 0,1,0,1,0,0,0,0,0\right)^\top,\;\; 
\vect\left( \begin{pmatrix} 0 & 0 & 0 \\ 0 & 0 & 0 \\ 0 & 0 & 1 \end{pmatrix} \right) = \left( 0,0,0,0,0,0,0,0,1\right)^\top. \nonumber
\end{align}
\end{example} 

\vspace{0.05in}

\noindent We now rewrite \eqref{eqn:sub_probH} as an equivalent vector optimization model using Theorem \ref{thm:Amat}. Vectorizing the individual terms in \eqref{eqn:sub_probH}, using the simple identity $\langle \bX, \bY \rangle = \vect(\bX)^\top \vect(\bY)$, the matrix optimization model is equivalent to, 
\begin{equation}
\min_{\bz \in \R^d}\; \left(\bB^\top \vect(\bnabla F(\bX)) \right)^\top \bz + \frac{1}{2}\bz^\top \left( \bB^\top \bm{A}(\bX) \bB \right)\bz + \frac{M}{6} \|\bz\|^3 , \nonumber 
\end{equation}
where the last term is rewritten using the fact, 
\[
\| \bH \|_F = \| \vect(\bH)\| = \| \bB\bz|\ = \sqrt{ \bz^\top \bB^\top \bB \bz } = \|\bz\|. 
\]
Thus, the vectorized matrix optimization subproblem is, 
\begin{equation}\label{eqn:vec_subprob}
\min_{\bz \in \R^d}\; \bb_X^\top \bz+ \frac{1}{2}\bz^\top \bA_X \bz + \frac{M}{6}\|\bz\|^3,
\end{equation}
where: 
\begin{equation}\label{eqn:b_X}
\bb_X:= \bB^\top  \vect(\bnabla F(\bX)),
\end{equation}
\begin{equation}\label{eqn:A_X}
\bA_X :=  \bB^\top \bA(\bX) \bB.
\end{equation}

Let $\bz_X$ denote the global optimal solution to \eqref{eqn:vec_subprob}. Upon solving \eqref{eqn:vec_subprob}, we obtain the vectorized global optimal solution to \eqref{eqn:sub_probH}, $\vect(\bH_X)$, as $\vect(\bH_X) = \bB \bz_X$. Note $\bz_X$ uniquely maps to $\vect(\bH_X)$ by the orthogonality of $\bB$. Additionally, $\bA_X \in \mathcal{S}^{d \times d}$. Since $D^2 F(\bX)[\bH_1, \bH_2]$ is symmetric in $\bH_1$ and $\bH_2$,
\[
\vect(\bH_1)^\top \bA(\bX) \vect(\bH_2) = \vect(\bH_2)^\top \bA(\bX) \vect(\bH_1), 
\]
for all $\bH_1, \bH_2 \in \cV$ which implies the symmetry of $\bA_X$. The orthogonality of $\bB$ also implies $\bA_X$ is unique though $\bA(\bX)$ is not necessarily uniquely defined from the vectorization. Another property of $\bA_X$ is its relationship to the Lipschitz Hessian of $F$.

\begin{lemma}\label{lem:mat_lipcond}
The function $F: \mathcal{V} \rightarrow \R$ has Lipschitz Hessian with parameter $L>0$ iff,
\[
 \| \bA_X - \bA_Y \|_2 \leq L \| \bX- \bY\|_F \; \text{ for all } \bX,\bY \in \mathcal{V}.
\]
\end{lemma}
\begin{proof}{Let $\bX, \bY \in \mathcal{V}$. Then, 
\begin{align}
\| D^2& F(\bX) - D^2 F(\bY) \| \nonumber \\
&= \underset{\bH_i \in \cV}{\max}\; \left\{ \big| D^2 F(\bX)[\bH_1, \bH_2] - D^2 F(\bY)[\bH_1, \bH_2] \big| \; \bigg| \;  \|\bH_1\|_F = \|\bH_2\|_F = 1 \right\} \nonumber \\
&= \underset{\bH_i \in \cV}{\max} \; \left\{ \big|\langle \bnabla^2 F(\bX)[\bH_1] - \bnabla^2 F(\bY)[\bH_1], \bH_2 \rangle \big| \; \bigg| \;  \|\bH_1\|_F = \|\bH_2\|_F = 1 \right\} \nonumber \\
&= \underset{\bH_i \in \cV}{\max} \; \left\{ \big| \vect(\bH_1)^\top \left( \bA(\bX) - \bA(\bY) \right)^\top \vect(\bH_2) \big| \; \bigg| \;  \|\vect(\bH_1)\| = \|\vect(\bH_2)\| = 1 \right\} \nonumber \\
&= \underset{\bz_i \in \R^d}{\max} \; \left\{ \big| \bz_1^\top \left( \bA_X- \bA_Y\right)^\top \bz_2 \big| \; \bigg| \;  \|\bz_1\| = \|\bz_2\| = 1 \right\} \nonumber \\
&= \; \|\bA_X - \bA_Y\|_2. \nonumber
\end{align}
}
\end{proof}

Additionally, $\bb_X$ and $\bA_X$ relate directly to the second-order $\epsilon$-stationary conditions for \eqref{eqn:crn_prob}.
\begin{lemma}\label{esp_lemma}
The point $\bX^* \in \cV \subseteq \R^{n\times n}$ is a second-order $\epsilon$-stationary point of \eqref{eqn:crn_prob} iff
\[
\|\bb_{\bX^*}\| \leq \epsilon \; \text{ and }\; \lambda_n(\bA_{\bX^*}) \geq -\sqrt{\epsilon}. 
\]
\end{lemma}

\begin{proof}\; By Theorem \ref{thm:Amat}, there exists an unique $\hat{\bz} \in \R^d$ such that $\bB\hat{\bz} = \vect(\bnabla F(\bX^*)$. By the definition of $\bb_X$ and the triangle inequality for matrix-vector multiplication, 
\[
\|\bb_{\bX^*}\| = \|\hat{\bz}\|  = \|\bB\|_2 \cdot \|\hat{\bz}\| \geq \|\bB \hat{\bz}\| = \| \bnabla F(\bX^*)\|_F.
\]
The second direction $\|\bb_{\bX^*}\| \leq  \| \bnabla F(\bX^*)\|_F$ follows immediately. Thus, $\|\bb_{X^*}\| =  \| \bnabla F(\bX^*)\|_F$ and both directions used the fact $\|\bB\|_2 = \|\bB^\top\|_2 = 1$. By Theorem \ref{thm:Amat} and the definitions of $\bA(\bX)$ and $\bA_X$, 
\begin{align}
\langle \bnabla^2 F(\bX)[\bH], \bH \rangle  &\geq - \sqrt{\epsilon} \| \bH\|_F^2,\;\; \forall \bH \in \cV \nonumber \\
\iff \vect(\bH)^\top \bA(\bX) \vect(\bH) &\geq -\sqrt{\epsilon} \| \vect(\bH)\|_F^2,\;\; \forall \bH \in \cV  \nonumber \\
\iff  \bz^\top \bA_X \bz &\geq -\sqrt{\epsilon}\|\bz\|^2, \;\; \forall \bz \in \R^d. \nonumber 
\end{align}
Hence, $\bX^*$ satisfies the secondary-order condition if and only if $\lambda_n\left( \bA_{\bX^*}\right)  \geq -\sqrt{\epsilon}$.\hfill
\end{proof}

We now connect \eqref{eqn:sub_probH} to the exact updating step for the cubic-regularized Newton method. For $M>0$, let, 
	\begin{equation}\label{eqn:subprob}
		\bT_{M}(\bX) \in \underset{{\bY \in \cV}}{\text{Argmin}}\left[ \langle \bnabla F(\bX), \bY-\bX \rangle + \frac{1}{2} \langle \bnabla^2 F(\bX)[\bY-\bX], \bY-\bX \rangle + \frac{M}{6} \|\bY-\bX\|_F^3 \right]
	\end{equation}
We note $\bT_M(\bX)$ and $\bH_X$ can be related by the equation, 
\begin{equation}\label{eqn:TM_H_relation}
\bT_M(\bX) = \bX + \bH_X. 
\end{equation}
Investigating \eqref{eqn:subprob} further, the first-order optimality condition for the problem requires, 
\begin{equation}\label{eqn:1stOrderOC}
\bnabla F(\bX) + \bnabla^2F(\bX)[\bT_M(\bX)-\bX] + \frac{M}{2}(\bT_M(\bX)-\bX)\|\bT_M(\bX)-\bX\|_F = 0.
\end{equation}
From this condition, we obtain the following useful bound. 
\begin{lemma}\label{lem:grad_bdd}
If $\bT_M(\bX)$ is well-defined, i.e. it exists and is bounded, then for all $\bX \in \mathcal{V}$, 
\begin{equation}\label{eqn:Nes_Lemma3}
\|\bnabla F(\bT_M(\bX))\|_F \leq \frac{1}{2}(M+L) \|\bH_X\|_F^2.
\end{equation}
\end{lemma}
\vspace{-0.05in}
\begin{proof}\; Let $\bX \in \cV$; from \eqref{eqn:1stOrderOC} it follows, 
\[
\|\bnabla F(\bX) + \bnabla^2F(\bX)[\bT_M(\bX)-\bX]\|_F = \frac{M}{2} \|\bT_M(\bX) - \bX\|_F^2, 
\]
and by \eqref{eqn:lip_ineq_1}, 
\[
\| \bnabla F(\bT_M(\bX)) - \bnabla F(\bX) - \bnabla^2 F(\bX)[\bT_M(\bX) - \bX ] \|_F \leq \frac{L}{2}\|\bT_M(\bX) - \bX\|_F^2.
\]
Replacing $\bT_M(\bX) - \bX$ with $\bH_X$ by \eqref{eqn:TM_H_relation}, adding the new inequalities, and applying the triangle inequality yields the result.\hfill
\end{proof}

Lemmas \ref{esp_lemma} and \ref{lem:grad_bdd} enable a potential function to be defined which directly mirrors the potential function in Section 3 of \cite{Nes06} where $\bb_X$ and $\bA_X$ replace the gradient and Hessian of a real-valued function defined on $\R^n$. Let, 
\begin{equation}\label{eqn:potential_func}
\mu_M(\bX) := \max \left\{\sqrt{\frac{2}{L+M}\|\bb_X\|_2},\;\;\; \frac{-2}{2L + M} \lambda_n(\bA_X)   \right\}.
\end{equation}
By construction $\mu_M(\bX)$ relates directly to the second-order $\epsilon$-stationality conditions. To be precise, if $\mu_M(\bX) \leq \max\left\{\sqrt{\frac{2\epsilon}{L+M}},\; \frac{-2\epsilon}{2L + M}\right\}$, then $\bX$ is a second-order $\epsilon$-stationary point. The next lemma relates $\mu_M(\bX)$ and the norm of the solution to the vectorized subproblem \eqref{eqn:vec_subprob}.

\begin{lemma}\label{lem:potential_func_bdd}
If $\bz_X$ is a global optimal solution to \eqref{eqn:vec_subprob},
then $\mu_M(\bT_M(\bX)) \leq \|\bz_X\|$. 
\end{lemma}
\begin{proof}\; From Lemma \ref{lem:grad_bdd} and the definition of $\bb_{\bX}$ we see, 
\begin{align}
\| \bb_{\bT_M(\bX)}\| &= \| \bB^\top \vect(\bnabla F(\bT_M(\bX)))\| \nonumber \\
				        &\leq \|\bB^\top\|_2 \cdot \|\bnabla F(\bT_M(\bX))\|_F \nonumber \\
					  &\leq \frac{1}{2}\left( L + M \right) \|\bH_X\|_F^2 \nonumber \\
					  &=\frac{1}{2}\left( L + M \right)\|\bz_X\|^2, \nonumber
\end{align}
where the last equality follows from the fact $\vect(\bH_X) = \bB \bz_X$. By Proposition 1 in \cite{Nes06}, 
\begin{equation}\label{eqn:prop1}
\bA_X + \frac{M}{2}\|\bz_X\| \bm{I} \succeq 0.
\end{equation}
Thus, with the Hessian of $F$ Lipschitz continuous, Lemma 
\ref{lem:mat_lipcond} yields, 
\begin{align}
\bA_{\bT_{M}(\bX)} &\succeq \bA_{X} - L \|\bH_X\|_F \bm{I} \nonumber  \\
			       &= \bA_{\bX} - L \|\bz_X\|_2 \bm{I} \nonumber \\
			       &\succeq - \left( \frac{M}{2} + L \right) \|\bz_X\| \bm{I}, \nonumber
\end{align}
where the last relation follows from \eqref{eqn:prop1}. Therefore, 
\[
\lambda_n(\bA_{\bT_M(\bX)}) \geq -  \left( \frac{M}{2} + L \right) \|\bz_X\|. 
\]
\hfill 
\vspace{-0.1in}
\end{proof}

With these results in hand, we prove Theorem \ref{thm:global_convergence_CRN}. The proof combines the arguments presented in \cite{Nes06,Jun18}, but the argument follows from the lemmas relating the matrix subproblem to the vectorized subproblem \\

\noindent{ \it Proof of Theorem \ref{thm:global_convergence_CRN}.}\;
By Lemma \ref{lem:hess_lip} and \eqref{eqn:TM_H_relation}, 
\begin{align}
F(\bX_{k+1}) & \leq F(\bX_k) + \langle F(\bX_k), \bH_{\bX_k} \rangle + \frac{1}{2} \langle \bnabla^2 F(\bX_k)[\bH_{\bX_k}], \bH_{\bX_k} \rangle + \frac{L}{6}\| \bH_{\bX_k}\|_F^3 \nonumber \\
			  &= F(\bX_k) + \bb_{\bX_k}^\top \bz_X + \frac{1}{2} \bz_{\bX_k}^\top \bA_{\bX_k} \bz_{\bX_k} + \frac{L}{6} \|\bz_{\bX_k}\|^3 \nonumber \\
			  &= F(\bX_k) + \bb_{\bX_k}^\top \bz_X + \frac{1}{2} \bz_{\bX_k}^\top \bA_{\bX_k} \bz_{\bX_k} + \left(\frac{L}{6}-\frac{M_k}{2} \right) \| \bz_{\bX_k}\|^3  - \bb_{\bX_k}^\top \bz_{\bX_k}  -  \bz_{\bX_k}^\top \bA_{\bX_k} \bz_{\bX_k} \nonumber \\
			  &= F(\bX_k) - \frac{1}{2} \bz_{\bX_k}^\top \bA_{\bX_k} \bz_{\bX_k}  + \left(\frac{L}{6}-\frac{M_k}{2} \right) \|\bz_{\bX_k}\|^3 \nonumber,
\end{align}
where the second equality comes from the optimality conditions of \eqref{eqn:vec_subprob}. Then, 
\[
F(\bX_{k+1}) \leq F(\bX_k) - \frac{1}{2} \bz_{\bX_k}^\top \left( \bA_{\bX_k} + \frac{M}{2}\|\bz_{\bX_k}\| \bm{I} \right) \bz_{\bX_k} - \left( \frac{M_k}{4} - \frac{L}{6} \right)\|\bz_{\bX_k}\|^3. 
\]
If $M_k > 2L/3$, then it follows $M/4 > L/6$ and it follows $F(\bX_{k+1}) \leq F(\bX_k)$ by \eqref{eqn:prop1}, and the inequality will be strict provided $\bz_{\bX_k} \neq 0$. Hence, letting $M_k \geq L$, we have, 
\[
F(\bX_{k+1}) - F(\bX_k)  \leq -\frac{L}{12}\|\bz_{\bX_k}\|^3. 
\]
Adding this inequality up from $k=0$ to $K-1$ we have, 
\[
\frac{L}{12} \sum_{k=0}^{K-1} \|\bz_{X_k}\|^3 \leq F(\bX_0) - F^*.
\]
Therefore, it follows 
\[
\sum_{k=0}^{\infty} \|\bz_{\bX_k}\|^3 \leq \frac{12 (F(\bX_0) - F^*)}{L},
\]
From Lemma \ref{lem:potential_func_bdd} we have $\mu_{M_k}(\bX_{k+1}) \leq \|\bz_{X_k}\|$ for all $k$, and additionally for any $k \geq 0$, 
\[
\frac{3}{4} \mu_{L}(\bX_{k+1}) \leq \mu_{2L}(\bX_{k+1})  \leq \mu_{M_k}( \bX_{k+1}) \leq \|\bz_{\bX_k}\|.
\]
Therefore, $\lim_{k\rightarrow \infty} \mu_{L}(\bX_{k+1})  = 0$ since $ \bz_{X_k}$ must limit to the zero vector, and, 
\[
\frac{L}{12} \sum_{k=0}^{K-1} \|\bz_{X_k}\|^3 \geq \frac{L}{12} \sum_{k=0}^{K-1} \left( \frac{3}{4}\right)^3 \mu_L(\bX_{k+1})^3 \geq \frac{L}{12} \left( \frac{3}{4}\right)^3  K \left( \min_{1 \leq k \leq K} \; \mu_L(\bX_k) \right)^3, 
\]
which implies,
\[
\min_{1 \leq k \leq K} \; \mu_L(\bX_k) \leq \frac{8}{3} \left( \frac{3}{2} \cdot \frac{F(\bX_0) - F^*}{KL} \right)^{1/3}.
\]
$\hfill \square$ 

We now prove the superlinear convergence rate of Algorithm \ref{alg:CRN} near non-degenerate local minimum described by Theorem \ref{thm:superlinear_CRN}. With our proven lemmas, the proof differs slightly from the argument provided for Theorem 3 in \cite{Nes06}. We provide the necessary links to allow their argument to hold. First, let,
\[
\delta_k = \frac{L \|\bb_{\bX_k}\|}{\lambda_{\min}^2(\bm{A}_{\bX_k})}, 
\]
which is a mirror image to the definition in Theorem 3 of \cite{Nes06}. 
\begin{theorem}\label{thm:superlinear_CRN_a} If Algorithm \ref{alg:CRN} is initialized sufficiently close to a non-degenerate local minimum $\bX^* \in \cV$, then the sequence $\{\bX_k\}$ generated by Algorithm \ref{alg:CRN} will convergence quadratically to $\bX^*$. 
\end{theorem}
\begin{proof}{
Assume $\bX_k \in \cV$ such that $\|D^2F(\bX_k)\|>0$ and $\delta_k \leq \frac{1}{4}$ for some $k \geq 0$. Then, by the characterization of the optimal solutions of \eqref{eqn:vec_subprob} in Theorem 3.1 of \cite{Ct11a} and the definitions of $\bH_{\bX_k}$ and $\bz_{\bX_k}$, 
\[
\|\bH_{\bX_k}\|_F = \|\bz_{\bX_k}|| = \|\left( \bm{A}_{\bX_k} + \frac{M_k}{2}\|\bz_{\bX_k}\| \bm{I} \right)^{-1} \bb_{\bX_k}\| \leq \frac{\|\bb_{\bX_k}\|}{\lambda_{\min}(\bA_{\bX_k})}.
\]
By Lemma \ref{lem:mat_lipcond}, we know $\bm{A}_{\bX_{k+1}} \succeq \bm{A}_{\bX_{k}} - \|z_{\bX_k}\| L \bm{I}$. Thus, by 
\[
\lambda_{\min}(\bm{A}_{\bX_{k+1}}) \; \geq \; \lambda_{\min}(\bm{A}_{\bX_{k}}) - L \|\bz_{\bX_k}\| \; \geq \; \lambda_{\min}(\bm{A}_{\bX_{k}}) - \frac{L\|\bb_{\bX_k}\|}{\lambda_{\min}(\bA_{\bX_k})} \; = \; (1-\delta_k) \lambda_{\min}(\bm{A}_{\bX_k}).
\]
Hence, since $\delta_k \in (0,1/4)$, it follows $\bm{A}_{\bX_{k+1}} \succ \bm{0}$ and $\delta_{k+1}$ is well-defined. By Lemma \ref{lem:grad_bdd} and the fact $M_k \leq 2L$, by the same argument in \cite{Nes06} we have, 
\[
\delta_{k+1} = \frac{L\|\bb_{\bX_{k+1}}\|}{\lambda_{\min}^2(\bm{A}_{\bX_{k+1}})} 
\leq \frac{3L^2 \|\bz_{\bX_k}\|^2}{2\lambda_{\min}^2(\bm{A}_{\bX_{k+1}})} \leq 
\frac{3L^2 \|\bb_{\bX_k}\|^2}{2\lambda_{\min}^4(\bm{A}_{\bX_{k+1}})(1-\delta_k)^2} = \frac{3}{2}\left( \frac{\delta_k}{1-\delta_k}   \right)^2   \leq 
\frac{8}{3} \delta_k^2. 
\]
Therefore, $\delta_{k+1} \leq \frac{1}{4}$ and the remainder of the proof given by Nesterov and Polyak in \cite{Nes06} follow directly provided the smallest eigenvalues of $\bm{A}_{\bX}$ is a continuous function of $\bX$, which is a consequence of Lemma \ref{lem:mat_lipcond}. Therefore, by part 3 of Theorem 3 in \cite{Nes06}, we have the whole sequence generated by Algorithm \ref{alg:CRN} converges quadratically to a local minimum point.}
\end{proof}


\section{Appendix D: Alternative Fair TME Formulations}\label{sec:extra_FairTME_models}
Another approach to solving the Fair TME model \eqref{eqn:fair_TME} is to utilize the matrix function $\bG(\bR) = e^{\bR}$ and rewrite \eqref{eqn:fair_TME} as the unconstrained model, 
\begin{align}\label{eqn:fair_TME_uncon_exp}
\min&\;\;\; \mu_1 \sum_{j=1}^{r} \cE_j(e^{\bR}) + \frac{\mu_2}{2}\sum_{i=1}^{r-1}\; \sum_{j=i+1}^{r} \left( \cE_i(e^{\bR}) - \cE_j(e^{\bR}) \right)^2 \\
\text{s.t.}&\;\; \bR \in \mathcal{S}^{p\times p}.\nonumber 
\end{align}

Both reformulations of \eqref{eqn:fair_TME}, \eqref{eqn:fair_TME_uncon_exp} and \eqref{eqn:fair_TME_uncon_sq}, solve the Fair TME model and each has its respective pros and cons. We note \eqref{eqn:fair_TME_uncon_exp} generally converges in fewer iterations due to the bijectivity of the matrix exponential; however, the cost of computing and/or approximating derivatives when solving  \eqref{eqn:fair_TME_uncon_exp} is more expensive and overflow issues are possible. On the other hand, \eqref{eqn:fair_TME_uncon_sq} has cheaper derivative computations and little risk of numerical overflow however more iterations of Algorithm \ref{alg:CRN} are often required to obtain second-order stationary points. Therefore, in terms of numerical implementation we advocate for \eqref{eqn:fair_TME_uncon_sq} and follow this approach in our numerical experiments in Section \ref{sec:experiments}. 


\end{document}